[1994/06/01]

\documentclass[12pt,reqno,intlimits,english,centertags]{amsart}
\usepackage[T1]{fontenc}

\usepackage{ifthen,calc}
\usepackage{babel}
\usepackage{xspace}
\usepackage{amsmath}
\usepackage[notref,notcite,final]{showkeys}

\usepackage{damacroenv}
\usepackage{lmodern}

\newcommand{\abs}[1]{\lvert#1\rvert}
\newcommand{\Abs}[1]{\left|#1\right|}

\DeclareMathOperator{\Lapl}{\Delta}
\newcommand{\norm}[1]{\lVert#1\rVert}

\newcommand{\norma}[2]{\norm{#1}_{#2}}

\newcommand{\SpDim}{N}
\newcommand{\numberspacefont}{\boldsymbol}
\newcommand{\R}{\numberspacefont{R}}
\newcommand{\RN}{\R^{\SpDim}}
\newcommand{\N}{\numberspacefont{N}}
\newcommand{\Z}{\numberspacefont{Z}}
\newcommand{\Pposbase}[3][*]{\ifthenelse{\equal{#1}{*}}%
{(#2)_{#3}}{\left(#2\right)_{#3}}}
\newcommand{\ppos}[1]{\Pposbase[*]{#1}{+}}

\newcommand{\di}{\,\textup{\textmd{d}}}
\newcommand{\eps}{\varepsilon}

\newcommand{\unk}{u}
\newcommand{\unkii}{v}
\newcommand{\unkiii}{w}

\newcommand{\wgt}{\omega}
\newcommand{\dgw}{d_{\wgt}}
\newcommand{\pder}[2]{\frac{\partial #1}{\partial #2}}
\newcommand{\fkf}{\Lambda_p}
\newcommand{\msw}{\mu_{\wgt}}
\newcommand{\dif}{D}
\newcommand{\floor}[1]{\lfloor#1\rfloor}
\newcommand{\mcut}{m}
\newcommand{\Mcut}{M}
\newcommand{\gnor}{\dif}
\newcommand{\dcf}{\psi}
\newcommand{\slt}{T}
\newcommand{\rdf}{\mathcal{R}}

\begin{document}

\title
{Asymptotic estimates for the $p$-Laplacian on infinite graphs\\with decaying initial data}%
\author{Daniele Andreucci}
\address{Department of Basic and Applied Sciences for Engineering\\Sapienza University of Rome, Italy}
\email{daniele.andreucci@sbai.uniroma1.it}
\thanks{The first author is member of the Gruppo Nazionale
   per la Fisica Matematica (GNFM) of the Istituto Nazionale di Alta Matematica
   (INdAM)}
\author{Anatoli F. Tedeev}
\address{South Mathematical Institute of VSC RAS\\Vladikavkaz, Russian Federation}
\email{a\_tedeev@yahoo.com}
\thanks{The second author was supported by Sapienza Grant C26V17KBT3}


\begin{abstract}
  We consider the Cauchy problem for the evolutive discrete $p$-Laplacian in infinite graphs, with initial data decaying at infinity.
  We prove optimal sup and gradient bounds for nonnegative solutions, when the initial data has finite mass, and also sharp evaluation for the confinement of mass, i.e., the effective speed of propagation. We provide estimates for some moments of the solution, defined using the distance from a given vertex.

  Our technique relies on suitable inequalities of Faber-Krahn type, and looks at the local theory of continuous nonlinear partial differential equations. As it is known, however, not all of this approach can have a direct counterpart in graphs.  A basic tool here is a result connecting the supremum of the solution at a given positive time with the measure of its level sets at previous times.

  We also consider the case of slowly decaying initial data, where the total mass is infinite.
\end{abstract}

\maketitle


\section{Introduction}\label{s:intro}

We consider nonnegative solutions to the Cauchy problem for discrete
degenerate parabolic equations
\begin{alignat}{2}
  \label{eq:pde}
  \pder{\unk}{t}(x,t)
  -
  \Lapl_{p}
  \unk (x,t)
  &=
  0
  \,,
  &\qquad&
  x\in V\,, t>0
  \,,
  \\
  \label{eq:init}
  \unk(x,0)
  &=
  \unk_{0}(x)
  \ge 0
  \,,
  &\qquad&
  x\in V
  \,.
\end{alignat}
Here $V$ is the set of vertices of the graph $G(V,E)$ with edge set $E\subset V\times V$ and weight $\wgt$, and
\begin{equation*}
  \Lapl_{p} u(x,t)
  =
  \frac{1}{\dgw(x)}
  \sum_{y\in V}
  \abs{\unk(y)-\unk(x)}^{p-2}
  (\unk(y)-\unk(x))
  \wgt(x,y)
  \,.
\end{equation*}
We assume that the graph $G$ is simple, undirected, infinite, connected with locally finite degree
\begin{equation*}
  \dgw(x)
  =
  \sum_{y\sim x}
  \wgt(x,y)
  \,,
\end{equation*}
where we write $y\sim x$ if and only if $\{x,y\}\in E$.
Here the weight $\wgt:V\times V\to [0,+\infty)$ is symmetric, i.e.,
$\wgt(x,y)=\wgt(y,x)$, and is strictly positive if and only if
$y\sim x$; then $\wgt(x,x)=0$ for $x\in V$.

We assume also that $p>2$ and that $\unk_{0}$ is nonnegative; further assumptions on $\unk_{0}$ will be stated below.

We prove sharp sup bounds for large times of solutions corresponding
to finite mass initial data; in order to prove the bound from below we
find an optimal estimate for the effective speed of propagation of
mass.  We also determine the stabilization rate for data exhibiting
slow decay `at infinity', in a suitable sense. To the best of our
knowledge such results are new in the framework of discrete nonlinear
diffusion equations on graphs.
\\
We apply an alternative approach, more local than the one in
\cite{Mugnolo:2013}, \cite{Hua:Mugnolo:2015} where the global
arguments of semigroup theory are extended to graphs, actually in a
more general setting which is out of the scope of this paper. We
comment below in the Introduction on the inherent difficulty and even
partial unfeasibility of a local approach in graphs. It is therefore
an interesting and not trivial problem to understand how much of this
body of techniques can be used in this environment. This paper can be
seen as a cross section of this effort; specifically we look at the
interplay between spread of mass and sup estimates, following ideas
coming from the theory of continuous partial differential equations,
with the differences required by the discrete character of graphs.

We recall the following notation: for any $R\in \N$, we let
\begin{equation*}
  B_{R}(x_{0})
  =
  \{
  x\in V
  \mid
  d(x,x_{0})
  \le R
  \}
  \,.
\end{equation*}
Here $d$ is the standard combinatorial distance in $G$ so that $d$ only takes integral values.
For any $f:V\to \R$ we set for all $q\ge 1$, $U\subset V$
\begin{gather*}
  \norma{f}{\ell^{q}(U)}^{q}
  =
  \sum_{x\in U}
  \abs{f(x)}^{q}
  \dgw(x)
  \,,
  \quad
  \norma{f}{\ell^{\infty}(U)}
  =
  \sup_{x\in U} \abs{f(x)}
  \,,
  \\
  \msw(U)
  =
  \sum_{x\in U}
  \dgw(x)
  \,.
\end{gather*}
All the infinite sums in this paper are absolutely convergent.  In the
following we always assume, unless explicitly noted, that all balls
are centered at a given fixed $x_{0}\in V$ and we write
$B_{R}(x_{0})=B_{R}$. We denote generic constants depending on the
parameters of the problem by $\gamma$ (large constants), $\gamma_{0}$
(small constants). We also set for all $A\subset V$
\begin{equation*}
  \chi_{A}(x)
  =
  1
  \,,
  \quad
  x\in A
  \,;
  \qquad
  \chi_{A}(x)
  =
  0
  \,,
  \quad
  x\not\in A
  \,.
\end{equation*}

\begin{definition}
  \label{d:fk}
  We say that $G$ satisfies a global Faber-Krahn inequality for a
  given $p>1$ and function $\fkf:(0,+\infty)\to(0,+\infty)$ if for any
  $v>0$ and any finite subset $U\subset V$ with $\msw(U)=v$ we have
  \begin{equation}
    \label{eq:fk}
    \fkf(v)
    \sum_{x\in U}
    \abs{f(x)}^{p}
    \dgw(x)
    \le
    \sum_{x,y\in (U)_{1}}
    \abs{f(y)-f(x)}^{p}
    \wgt(x,y)
    \,,
  \end{equation}
  for all $f:V\to \R$ such that $f(x)=0$ if $x\not \in U$; here
  \begin{equation*}
    (U)_{1}
    =
    \{
    x\in V
    \mid
    d(x,U)
    \le 1
    \}
    \,.
  \end{equation*}
\end{definition}

We assume throughout that $\fkf\in C(0,+\infty)$ is decreasing and
that two suitable positive constants $\SpDim$, $\omega$ exists such that
\begin{align}
  \label{eq:fkf_nd}
  v&\mapsto
     \fkf(v)^{-1}
     v^{-\frac{p}{\SpDim}}
     \,,
     \quad
     v>0
     \,,
     \quad
     \text{is nondecreasing;}
  \\
  \label{eq:fkf_ni}
  v&\mapsto
     \fkf(v)^{-1}
     v^{-\omega}
     \,,
     \quad
     v>0
     \,,
     \quad
     \text{is nonincreasing.}
\end{align}
An important class of functions in our approach is given by 
\begin{equation}
  \label{eq:dcf_def}
  \dcf_{r}(s)
  =
  s^{\frac{p-2}{r}}
  \fkf(s^{-1})
  \,,
  \qquad
  s>0
  \,,
\end{equation}
for each fixed $r\ge 1$. They, or more exactly their inverses, give the correct time-space scaling for
the equation \eqref{eq:pde}, see for example Theorems \ref{t:l1} and \ref{p:bbl} below.

If we make the additional assumption that for some constant $c>0$
\begin{equation}
  \label{eq:fkf_bbl}
  \fkf(v)
  \ge
  c
  \rdf(v)^{-p}
  \,,
  \qquad
  v>0
  \,,
\end{equation}
where $\rdf:(0,+\infty)\to (0,+\infty)$ is such that
$\msw(B_{\rdf(v)})=v$, we may connect $\dcf_{1}$ to the measure of a
ball in $G$. This in turn allows us to prove sharpness of our
$\ell^{1}$--$\ell^{\infty}$ estimate. Property \eqref{eq:fkf_bbl} is
rather natural. For instance it is known to hold for the explicit examples in
Subsection~\ref{s:examples}, to which we refer for
implementations of our results in some concrete relevant cases.

\begin{remark}
  \label{r:spdim}
  The constant $\SpDim$ in \eqref{eq:fkf_nd} has no intrinsic meaning
  in this paper, and it is employed here only with the purpose of
  making easier the comparison with the case of standard regular
  graphs $\Z^{\SpDim}$, where $\fkf(v)=\gamma_{0}v^{-p/\SpDim}$, see Subsection~\ref{s:examples}.
\end{remark}

\begin{remark}
  \label{r:fk_below}
  Let $x\in V$ and choose $U=\{x\}$, $f=\chi_{U}$ in \eqref{eq:fk}, which then yields
  \begin{equation}
    \label{eq:rem_fk}
    \fkf(\dgw(x))
    \dgw(x)
    \le
    2\dgw(x)
    \,.
  \end{equation}
  Since $\fkf$ is decreasing by assumption we infer
  \begin{equation}
    \label{eq:rem_fk_bound}
    \dgw(x)\ge \fkf^{(-1)}(2)
    \,.
  \end{equation}
  A remark in this connection is perhaps in order: clearly according
  to its definition the Faber-Krahn function $\fkf(v)$ is defined for
  uniformly positive $v$ according to \eqref{eq:rem_fk_bound}, so that
  \eqref{eq:fkf_nd}, \eqref{eq:fkf_ni} should be assumed for such
  $v$. Aiming at a technically streamlined framework, we extend $\fkf$
  to all $v>0$, while easily preserving the latter
  assumptions. However one can check that for large times, $\fkf$ is
  evaluated at large arguments in our results, which are thus
  independent of this extension.
\end{remark}

\begin{remark}
  \label{r:lp_scale}
  A consequence of \eqref{eq:rem_fk_bound} is that any bound in $\ell^{q}(V)$ yields immediately a uniform pointwise bound: if $\unkii\in\ell^{q}(V)$,
  \begin{equation}
    \label{eq:lp_linf}
    \abs{\unkii(z)}^{q}
    \le
    \abs{\unkii(z)}^{q}\frac{\dgw(z)}{\fkf^{(-1)}(2)}
    \le
    \frac{1}{\fkf^{(-1)}(2)}
    \norma{\unkii}{\ell^{q}(V)}^{q}
    \,,
    \quad
    z\in V
    \,.
  \end{equation}
  In turn this implies that $\ell^{p}(V)\subset\ell^{q}(V)$ if $p<q$, since
  \begin{equation*}
    \sum_{x\in V}
    \abs{f(x)}^{q}
    \dgw(x)
    \le
    M^{q-p}
    \sum_{x\in V}
    \abs{f(x)}^{p}
    \dgw(x)
    \,,
  \end{equation*}
  for a suitable $M$ as in \eqref{eq:lp_linf}.
\end{remark}

\begin{definition}
  \label{d:sol}
  We say that $\unk\in L^{\infty}(0,T; \ell^{r}(V))$ is a solution to
  \eqref{eq:pde} if $\unk(x)\in C^{1}([0,T])$ for every
  $x\in V$ and $\unk$ satisfies \eqref{eq:pde} in the classical
  pointwise sense.
  \\
  A solution to \eqref{eq:pde}--\eqref{eq:init} also is required to
  take the initial data prescribed by \eqref{eq:init}, for each
  $x\in V$.
\end{definition}

We refer the reader to \cite{Hua:Mugnolo:2015} for existence and
uniqueness of solutions. To make this paper more self-contained
however we sketch in Section~\ref{s:prelim} a proof of existence in
Proposition~\ref{p:ex} (in $\ell^{q}$, $q>1$, see Theorem~\ref{t:l1} for $q=1$),
and of uniqueness via comparison in Proposition~\ref{p:compare}.

Our first two results are typical of the local approach we pursue. All solutions we consider below are nonnegative. 
\begin{proposition}
  \label{p:linf_meas}
  Let $\unk:V\to \R$ be a  solution to \eqref{eq:pde},
  with $\unk\in L^{\infty}(0,T;\ell^{r}(V))$ for some $r\ge 1$. Then for all $x\in V$, $0<t<T$
  \begin{equation}
    \label{eq:linf_meas_n}
    \unk(x,t)
    \le
    k
    \,,
  \end{equation}
  provided $k>0$ satisfies for a suitable $\gamma_{0}(p,\SpDim)$
  \begin{equation}
    \label{eq:linf_meas_nn}
    k^{-1}
    t^{-\frac{1}{p-2}}
    \fkf\Big(\sup_{\frac{t}{4}<\tau<t}\msw(\{x\in V\mid \unk(x,\tau)> {k}/{2}\})\Big)^{-\frac{1}{p-2}}
    \le
    \gamma_{0}
    \,.
  \end{equation}
\end{proposition}

\begin{corollary}
  \label{co:linf_int}
  Under the assumptions in Proposition~\ref{p:linf_meas}, we have
  \begin{equation}
    \label{eq:linf_int_m}
    \unk(x,t)
    \le
    \gamma
    \sup_{0<\tau<t}\norma{\unk(\tau)}{\ell^{r}(V)}
    \big[\dcf_{r}^{(-1)}
    \big(
    t^{-1}
    \sup_{0<\tau<t}\norma{\unk(\tau)}{\ell^{r}(V)}^{-(p-2)}
    \big)
    \big]^{\frac{1}{r}}
    \,,
  \end{equation}
  for all $x\in V$, $0<t<T$.
  Here $\dcf_{r}^{(-1)}$ is the inverse function of $\dcf_{r}$ as defined in \eqref{eq:dcf_def}.
\end{corollary}

\begin{remark}
  \label{r:dcf}
  One can check easily using the fact that $\fkf$ is nonincreasing that
  \begin{equation*}
    a\mapsto a\dcf_{r}^{(-1)}(sa^{-(p-2)})^{\frac{1}{r}}
  \end{equation*}
  is nondecreasing in $a>0$ for each fixed $s>0$.
\end{remark}

Next Theorem follows directly from the estimates we stated above. Note
that conservation of mass in \eqref{eq:l1_n} was proved also in
\cite{Hua:Mugnolo:2015}, while the other estimates are new, as far as
we know.
\begin{theorem}
  \label{t:l1}
  Let $\unk_{0}\in\ell^{1}(V)$, $\unk_{0}\ge 0$. Then problem
  \eqref{eq:pde}--\eqref{eq:init} has a unique solution
  satisfying for all $t>0$
  \begin{align}
    \label{eq:l1_n}
    \norma{\unk(t)}{\ell^{1}(V)}
    &=
      \norma{\unk_{0}}{\ell^{1}(V)}
      \,,
    \\
    \label{eq:l1_nn}
    \norma{\unk(t)}{\ell^{\infty}(V)}
    &\le
      \gamma
      \norma{\unk_{0}}{\ell^{1}(V)}
      \dcf_{1}^{(-1)}
      \big(
      t^{-1}
      \norma{\unk_{0}}{\ell^{1}(V)}^{-(p-2)}
      \big)
      \,.
  \end{align}
  In addition $\unk$ satisfies
  \begin{multline}
    \label{eq:l1_nnn}
    \int_{0}^{t}
    \sum_{x,y\in V}
    \abs{\dif_{y}\unk(x,\tau)}^{p-1}
    \wgt(x,y)
    \di\tau
    \\
    \le
    \gamma
    t^{\frac{1}{p}}
    \norma{\unk_{0}}{\ell^{1}(V)}^{\frac{2(p-1)}{p}}
    \dcf_{1}^{(-1)}(t^{-1}\norma{\unk_{0}}{\ell^{1}(V)}^{-(p-2)})^{\frac{p-2}{p}}
    \,.
  \end{multline}
\end{theorem}

\begin{remark}
  \label{r:lp_est}
  We notice that one could exploit \eqref{eq:l1_n}, \eqref{eq:l1_nn}
  to derive trivially a bound of the integral in
  \eqref{eq:l1_nnn}. This is due of course to the fact that the
  $p$-laplacian in our setting is discrete, and it would not be
  possible in the framework of continuous partial differential
  equations. 

  Such a bound however is not sharp, and for example could not be
  used in the proof of Theorem~\ref{p:bbl}.

  In other instances where optimality is not needed we exploit a
  device similar to the one just described, relying on
  Remark~\ref{r:lp_scale}; see for example the proof of
  Lemma~\ref{l:cacc2}.
\end{remark}

So far our extension to graphs of methods and results of
continuous differential equations has been successful. However, in the
latter setting a standard device to prove optimality of the bound in
\eqref{eq:l1_nn} relies on the property of finite speed of propagation
(i.e., solutions with initially bounded support keep this feature for
all $t>0$). In the setting of graphs this property strikingly fails,
as shown in \cite{Hua:Mugnolo:2015}. As a technical but perhaps
worthwile side remark, we note that all the main ingredients in the
proof of finite speed of propagation (see \cite{Andreucci:Tedeev:1999},
\cite{Andreucci:Tedeev:2000}) seem to be available in graphs too:
embeddings as in \cite{Ostrovskii:2005}, Caccioppoli inequalities as in
Lemma~\ref{l:cacc} below, and of course iterative techniques as the
one displayed in the proof of Proposition~\ref{p:linf_meas}. The key
exception in this regard is the fact that full localization via an
infinite sequence of nested shrinking balls is clearly prohibited by
the discrete metric at hand. This is a point of marked difference with
the continuous setting.

Still we can prove sharpness of our $\ell^{1}$--$\ell^{\infty}$ bound
\eqref{eq:l1_nn} by means of the following result of confinement of
mass. By the same argument we can estimate also
a suitable moment of the solution, which is also a new result for
nonlinear diffusion in graphs, see Section~\ref{s:bbl}.
\begin{theorem}
  \label{p:bbl}
  Let $\unk_{0}\ge 0$ be finitely
  supported. Then for every $1>\eps >0$ there exists a $\varGamma>0$
  such that
  \begin{equation}
    \label{eq:bbl_n}
    \norma{\unk(t)}{\ell^{1}(B_{R})}
    \ge
    (1-\eps)
    \norma{\unk_{0}}{\ell^{1}(V)}
    \,,
    \qquad
    t>0
    \,,
  \end{equation}
  provided $B_{\floor{R/2}}$ contains the support of $\unk_{0}$, and $R$ is chosen so that
  \begin{equation}
    \label{eq:bbl_nn}
    R
    \ge
    \varGamma
    t^{\frac{1}{p}}
    \norma{\unk_{0}}{\ell^{1}(V)}^{\frac{p-2}{p}}
    \dcf_{1}^{(-1)}(t^{-1}\norma{\unk_{0}}{\ell^{1}(V)}^{-(p-2)})^{\frac{p-2}{p}}
    \ge 8
    \,.
  \end{equation}
  In addition, provided $R$ is chosen as in \eqref{eq:bbl_nn}, for $\eps=1/2$, and $\alpha\in (0,1)$,
  \begin{equation}
    \label{eq:bbl_p}
    \sum_{x\in V}
    d(x,x_{0})^{\alpha}
    \unk(x,t)
    \dgw(x)
    \le
    \gamma
    R^{\alpha}
    \norma{\unk_{0}}{\ell^{1}(V)}
    \,,
    \qquad
    t>0
    \,.
  \end{equation}
\end{theorem}
Next we exploit the estimate \eqref{eq:bbl_n}--\eqref{eq:bbl_nn} in order to show
that up to a change in the constant we can reverse the inequality in \eqref{eq:l1_nn},
proving at once the optimality of both results.

\begin{corollary}
  \label{p:bbl2}
  Under the assumptions in Theorem~\ref{p:bbl}, let in
  addition $\fkf$ satisfy \eqref{eq:fkf_bbl}. Then
  \begin{equation}
    \label{eq:bbl_nnn}
    \norma{\unk(t)}{\ell^{\infty}(V)}
    \ge
    \frac{\norma{\unk_{0}}{\ell^{1}(V)}}{2\msw(B_{R})}
    \ge
    \gamma_{0}
    \norma{\unk_{0}}{\ell^{1}(V)}
    \dcf_{1}^{(-1)}
    \big(
    t^{-1}
    \norma{\unk_{0}}{\ell^{1}(V)}^{-(p-2)}
    \big)
    \,,
  \end{equation}
  where $R$ is as in \eqref{eq:bbl_nn}, for $\eps=1/2$.
\end{corollary}

Clearly, owing to the comparison principle of
Proposition~\ref{p:compare}, results like those in \eqref{eq:bbl_n}
and \eqref{eq:bbl_nnn} may be proved even dropping the assumption
that $\unk_{0}$ is finitely supported; for the sake of brevity we omit
the details.

In order to state our last result we need to introduce the following
function, which essentially gives the correct scaling between time and
space in the case of slow decay initial data: for
$\unk_{0}\in\ell^{q}(V)\setminus\ell^{1}(V)$ for some $q>1$ set
\begin{equation}
  \label{eq:decay_fn}
  \slt_{\unk_{0}}(R,x_{0})
  =
  \Big[
  \frac{\norma{\unk_{0}}{\ell^{1}(B_{R}(x_{0}))}}{\norma{\unk_{0}}{\ell^{q}(V\setminus B_{R}(x_{0}))}^{q}}
  \Big]^{\frac{p-2}{q-1}}
  \,
  \fkf\bigg(
  \Big(
  \frac{
    \norma{\unk_{0}}{\ell^{1}(B_{R}(x_{0}))}
  }{
    \norma{\unk_{0}}{\ell^{q}(V\setminus B_{R}(x_{0}))}
  }
  \Big)^{\frac{q}{q-1}}
  \bigg)^{-1}
  \,,
\end{equation}
for $R\in \N$, $x_{0}\in V$. Clearly for each fixed $x_{0}$ the
function $\slt_{\unk_{0}}$ is nondecreasing in $R$ and $\slt_{\unk_{0}}(R,x_{0})\to +\infty$ as
$R\to\infty$. Conversely, $\slt_{\unk_{0}}(0,x_{0})$ may be positive. However it can
be easily seen that for any given $\eps>0$ there exists $x_{0}$ such
that $\slt_{\unk_{0}}(0,x_{0})<\eps$.

\begin{theorem}
  \label{t:decay}
  Let $\unk_{0}\in\ell^{q}(V)\setminus\ell^{1}(V)$ for some $q>1$. Then for all $t>0$, $x_{0}\in V$
  \begin{equation}
    \label{eq:decay_n}
    \norma{\unk(t)}{\ell^{\infty}(V)}
    \le
    \gamma
    \norma{\unk_{0}}{\ell^{1}(B_{R}(x_{0}))}
    \dcf_{1}^{(-1)}\big(t^{-1} \norma{\unk_{0}}{\ell^{1}(B_{R}(x_{0}))}^{-(p-2)}\big)
    \,,
  \end{equation}
  provided $R$ is chosen so that
  \begin{equation}
    \label{eq:decay_nn}
    t\le
    \slt_{\unk_{0}}(R,x_{0})
    \,,
  \end{equation}
  the optimal choice being of course the minimum $R=R(t)$ such that
  \eqref{eq:decay_nn} holds true.
\end{theorem}

Let us comment briefly on the existing literature on the non-linear
$p$-Laplacian in graphs. The papers \cite{Mugnolo:2013},
\cite{Hua:Mugnolo:2015}, deal with the Cauchy problem applying
techniques inspired from the theory of semigroups of continuous
differential operators. They consider a more general variety of
weighted graphs and operators than we do here, dealing e.g., with
existence, uniqueness, time regularity, possible extinction in a
finite time. However our results do not seem to be easily reached by
this approach. We also quote \cite{Keller:Mugnolo:2016} where a
connection between Cheeger constants and the eigenvalues of the
$p$-laplacian is drawn in a very flexible setting.
\\
Boundary problems on finite subgraphs are also considered in several
papers dealing with features like blow up or extinction; we
quote only \cite{Chung:Choi:2014} and \cite{Chung:Park:2017}.

The case of the discrete linear Laplacian where $p=2$ is more
classical, also for its connections with probability theory (see e.g.,
\cite{Andres:etal:2013} and references therein), and is often attacked
by means of suitable parallels with the theory of heat kernels in
manifolds. We quote \cite{Coulhon:Grigoryan:1998},
\cite{Barlow:Coulhon:Grigoryan:2001} where a connection is drawn
between properties of heat kernels, of graphs and Faber-Krahn
functions.
\\
In \cite{Lin:Wu:2017} heat kernels are used to study the blow up of
solutions to the Cauchy problem for a semilinear equation on a
possibly infinite graph.

The subject of diffusion in graphs is popular also owing to its
applicative interest. We refer the reader to \cite{Mugnolo:2013},
\cite{Elmoataz:Toutain:Tenbrinck:2015} and to the references therein
for more on this point.

Finally we recall the papers \cite{Bakry:etal:1995a},
\cite{Ostrovskii:2005} and books \cite{Chung:SGT}, \cite{Grigoryan:AG}
for basic information on functional analysis on graphs and manifolds.

We mention that in our setting it is still valid the argument in
\cite{Bonforte:Grillo:2007} showing that optimal decay rates imply suitable
embeddings.

Here we look essentially at the approach of \cite{DiBenedetto:Herrero:1989} and \cite{Andreucci:Tedeev:2015}.

The paper is organized as follows: Section~\ref{s:prelim} is devoted
to preliminary material.  Proposition~\ref{p:linf_meas} and its
Corollary~\ref{co:linf_int} are proved in Section~\ref{s:linf}, while
Section~\ref{s:l1} contains the proof of Theorem~\ref{t:l1} and
Section~\ref{s:bbl} deals with Theorem~\ref{p:bbl} and
Corollary~\ref{p:bbl2}. Finally Theorem~\ref{t:decay} is proved in
Section~\ref{s:decay}.

\subsection{Examples}
\label{s:examples}
1) As a first example we consider the case of the standard lattice
$G=\Z^{\SpDim}$, where one can take $\fkf(v)=\gamma_{0}v^{-p/\SpDim}$,
according to the results of \cite{Wang:Wang:1977},
\cite{Ostrovskii:2005}. This is the case where comparison with the
Cauchy problem for the continuous $p$-Laplacian is more straightforward. In this
case
\begin{equation}
  \label{eq:dcf_euc}
  \dcf_{r}(s)
  =
  \gamma_{0}
  s^{\frac{\SpDim(p-2)+pr}{\SpDim r}}
  \,,
  \qquad
  s>0
  \,, 
\end{equation}
and for example estimate \eqref{eq:l1_nn} becomes
\begin{equation}
  \label{eq:l1_nn_grid}
  \norma{\unk(t)}{\ell^{\infty}(V)}
  \le
  \gamma
  \norma{\unk_{0}}{\ell^{1}(V)}^{\frac{p}{\SpDim(p-2)+p}}
  t^{-\frac{\SpDim}{\SpDim(p-2)+p}}
  \,,
\end{equation}
while the critical radius for expansion of mass in \eqref{eq:bbl_nn} amounts to
\begin{equation}
  \label{eq:bbl_nn_grid}
  R\ge
  \gamma
  \norma{\unk_{0}}{\ell^{1}(V)}^{\frac{p-2}{\SpDim(p-2)+p}}
  t^{\frac{1}{\SpDim(p-2)+p}}
  \,.
\end{equation}
We remark that both results formally coincide with the corresponding
ones for the continuous $p$-Laplacian in $\RN$, see
\cite{DiBenedetto:Herrero:1989}.
\\
Next we apply Theorem~\ref{t:decay} to the following initial data: for
$x=(x_{1},\dots,x_{\SpDim})\in\Z^{\SpDim}$ set
$\unk_{0}(x)=(\abs{x_{1}}+\dots+\abs{x_{\SpDim}})^{-\alpha}$ for a
given $0<\alpha<\SpDim$. Let us write here $a(s)\simeq b(s)$ if
$\gamma_{0}a(s)\le b(s)\le \gamma a(s)$ for two constants independent
of $s$.  One can see that
\begin{equation*}
  \norma{\unk_{0}}{\ell^{1}(B_{R}(0))}
  \simeq
  R^{\SpDim-\alpha}
  \,;
  \qquad
  \norma{\unk_{0}}{\ell^{q}(V\setminus B_{R}(0))}
  \simeq
  R^{\SpDim-\alpha q}
  \,,
\end{equation*}
for all $q>\SpDim/\alpha$. Therefore in this case 
\begin{equation*}
  \slt_{\unk_{0}}(R,0)
  \simeq
  R^{\alpha(p-2)+p}
  \,,
\end{equation*}
and the estimate in \eqref{eq:decay_n} essentially amounts to the
decay rate $t^{-\alpha/(\alpha(p-2)+p)}$, which is the expected one in
view of the results of \cite{Tedeev:1991}.

2) One can treat also other examples of product graphs; for instance
if $H$ is a finite connected graph we let $G=H\times \Z^{\SpDim}$ and
recover results similar to the ones of the previous example.

3) All examples where the Faber-Krahn function is estimated for $p=2$
yield also examples in our case of $p>2$, as it follows from applying
H\"older's inequality; see e.g., \cite{Coulhon:Grigoryan:1998},
\cite{Barlow:Coulhon:Grigoryan:2001}.

\section{Preliminary material}
\label{s:prelim}

We use for $f:V\to \R$ the notation
\begin{equation*}
  \dif_{y}f(x)
  =
  f(y)-f(x)
  =
  -
  \dif_{x}f(y)
  \,,
  \qquad
  x\,,y\in V
  \,.
\end{equation*}

\subsection{Caccioppoli type inequalities}
\label{s:prelim_cacc}
\begin{lemma}
  \label{l:monot}
  Let $q>0$, $p>2$, $h\ge 0$, $\unk$, $\unkii:V\to\R$. Then for all $x$, $y\in V$
  \begin{multline}
    \label{eq:monot_n}
    \big(
    \abs{\dif_{y}\unk(x)}^{p-2}
    \dif_{y}\unk(x)
    -
    \abs{\dif_{y}\unkii(x)}^{p-2}
    \dif_{y}\unkii(x)
    \big)
    \dif_{y}\ppos{\unk(x)-\unkii(x)-h}^{q}
    \\
    \ge
    \gamma_{0}
    \Abs{
      \dif_{y}\ppos{\unk(x)-\unkii(x)-h}^{\frac{q-1+p}{p}}
    }^{p}
    \,.
  \end{multline}
\end{lemma}

\begin{proof}
  First we remark that we may assume $h=0$, by renaming
  $\tilde\unkii=\unkii+h$.  The corresponding version of
  \eqref{eq:monot_n} clearly holds true if
  $\dif_{y}\unk(x)=\dif_{y}\unkii(x)$. 

  If $\dif_{y}\unk(x)\not=\dif_{y}\unkii(x)$ the left hand side of
  \eqref{eq:monot_n} with $h=0$ can be written as, on appealing also
  to a classical elementary result in monotone operators, see
  \cite{DiB:dpe},
  \begin{multline}
    \label{eq:monot_i}
    \big(
    \abs{\dif_{y}\unk(x)}^{p-2}
    \dif_{y}\unk(x)
    -
    \abs{\dif_{y}\unkii(x)}^{p-2}
    \dif_{y}\unkii(x)
    \big)
    \dif_{y}(\unk(x)-\unkii(x))
    \,
    \mathcal{A}
    \\
    \ge
    \gamma_{0}(p)
    \abs{\dif_{y}(\unk(x)-\unkii(x))}^{p}
    \,
    \mathcal{A}
    \,,
  \end{multline}
  where we define
  \begin{equation*}
    \mathcal{A}
    =
    \frac{
      \dif_{y}\ppos{\unk(x)-\unkii(x)}^{q}
    }{
      \dif_{y}(\unk(x)-\unkii(x))
    }
    \ge 0
    \,.
  \end{equation*}

  On the other hand, we write the right hand side of \eqref{eq:monot_n} with $h=0$ as
  \begin{equation}
    \label{eq:monot_ii}
    \abs{\dif_{y} (\unk(x)-\unkii(x))}^{p}
    \,
    \mathcal{B}
    \,,
    \qquad
    \mathcal{B}
    :=
    \Abs{
      \frac{
        \dif_{y}\ppos{\unk(x)-\unkii(x)}^{\frac{q-1+p}{p}}
      }{
        \dif_{y}(\unk(x)-\unkii(x))
      }
    }^{p}
    \,.
  \end{equation}
  Therefore we have only to prove that
  $\mathcal{A}\ge \gamma_{0}\mathcal{B}$. Clearly in doing so we may
  assume without loss of generality that
  \begin{equation*}
    \unk(y)
    -
    \unkii(y)
    >
    \unk(x)
    -
    \unkii(x)
    \,.
  \end{equation*}
  Hence it is left to prove that
  \begin{multline}
    \label{eq:monot_iii}
    \big[
    \unk(y)-\unkii(y)
    -
    (\unk(x)-\unkii(x))
    \big]^{p-1}
    \big[
    \ppos{\unk(y)-\unkii(y)}^{q}
    -
    \ppos{\unk(x)-\unkii(x)}^{q}
    \big]
    \\
    \ge
    \gamma_{0}
    \Big[
    \ppos{\unk(y)-\unkii(y)}^{\frac{q-1+p}{p}}
    -
    \ppos{\unk(x)-\unkii(x)}^{\frac{q-1+p}{p}}
    \Big]^{p}
    \,.
  \end{multline}
  Denote
  \begin{equation*}
    a
    =
    \unk(y)
    -
    \unkii(y)
    \,,
    \qquad
    b
    =
    \unk(x)
    -
    \unkii(x)
    \,.
  \end{equation*}
  If $b\le 0$, \eqref{eq:monot_iii} is obviously satisfied with
  $\gamma_{0}=1$. If $b>0$, by H\"older's inequality we have
  \begin{multline}
    \label{eq:monot_iv}
    \big[
    a^{\frac{q-1+p}{p}}
    -
    b^{\frac{q-1+p}{p}}
    \big]^{p}
    =
    \Big[
    \frac{q-1+p}{p}
    \int_{b}^{a}
    s^{\frac{q-1}{p}}
    \di s
    \Big]^{p}
    \\
    \le
    \gamma(q,p)
    \Big[
    \int_{b}^{a}
    s^{q-1}
    \di s
    \Big]
    \Big[
    \int_{b}^{a}
    \di s
    \Big]^{p-1}
    \le
    \gamma(q,p)
    (a^{q}-b^{q})
    (a-b)^{p-1}
    \,,
  \end{multline}
  proving \eqref{eq:monot_iii} and concluding the proof.
\end{proof}

In the following all radii of balls in $G$ will be assumed to be natural numbers.
Let $R_{2}\ge R_{1}+1$, $R_{1}$, $R_{2}>0$; we define
the cutoff function $\zeta$ in $B_{R_{2}}(x_{0})$ by means of
\begin{alignat*}{2}
  \zeta(x)
  &=
  1
  \,,
  &\qquad&
  x\in B_{R_{1}}(x_{0})
  \,,
  \\
  \zeta(x)
  &=
  \frac{
    R_{2}
    -
    d(x,x_{0})
  }{
    R_{2}
    -
    R_{1}
  }
  \,,
  &\qquad&
  x\in B_{R_{2}}\setminus B_{R_{1}}(x_{0})
  \,,
  \\
  \zeta(x)
  &=
  0
  \,,
  &\qquad&
  x\not \in B_{R_{2}}(x_{0})
  \,.
\end{alignat*}
The function $\zeta$ is chosen so that
\begin{equation*}
  \abs{\dif_{y}\zeta(x)}
  =
  \abs{
    \zeta(y)
    -
    \zeta(x)
  }
  \le
  \frac{1}{R_{2}-R_{1}}
  \,,
  \qquad
  x\sim y
  \,.
\end{equation*}
For $\tau_{1}>\tau_{2}>0$ we also define the standard nonnegative cutoff function $\eta\in C^{1}(\R)$ such that
\begin{equation*}
  \eta(t)
  =0
  \,,
  \,\,\,
  t\ge \tau_{1}
  \,;
  \quad
  \eta(t)
  =
  0
  \,,
  \,\,\,
  t\le \tau_{2}
  \,;
  \quad
  0\le\eta'(t)\le \frac{2}{\tau_{1}-\tau_{2}}
  \,,
  \,\,\,
  t\in\R
  \,.
\end{equation*}

Our next Lemma is not used in the sequel; we present it here to substantiate our claim made in the Introduction that suitable local Caccioppoli type inequalities are available in the nonlinear setting, and also for its possible independent interest. The proof is somehow more complex than in the continuous case.
\begin{lemma}
  \label{l:cacc}
  Let $\unk$ be a solution of \eqref{eq:pde} in $V\times (0,T)$, $x_{0}\in V$. Then
  for $T>\tau_{1}>\tau_{2}>0$, $R_{2}>R_{1}+1$, $R_{1}>0$, $h>k>0$, $1>\theta>0$ we
  have
  \begin{equation}
    \label{eq:cacc_n}
    \begin{split}
      &\sup_{\tau_{1}<\tau<t}
      \sum_{x\in B_{R_{1}}(x_{0})}
      \ppos{\unk(x,\tau)-h}^{\theta+1}
      \zeta(x)^{p}
      \dgw(x)
      \\
      &\quad+
      \int_{\tau_{1}}^{t}
      \sum_{x\in B_{R_{1}}(x_{0}),y\in V}
      \Abs{
        \dif_{y}
        \ppos{\unk(x,\tau)-h}^{\frac{p+\theta-1}{p}}
      }^{p}
      \wgt(x,y)
      \di\tau
      \\
      &\qquad\le
      \frac{\gamma}{\tau_{1}-\tau_{2}}
      \int_{\tau_{2}}^{t}
      \sum_{x\in B_{R_{2}}(x_{0})}
      \ppos{\unk(x,\tau)-h}^{\theta+1}
      \dgw(x)
      \di\tau
      +
      \gamma
      A^{\frac{1}{p}}
      B^{\frac{p-1}{p}}
      +
      \gamma
      A
      \,,
    \end{split}
  \end{equation}
  where
  \begin{align*}
    A&=
       \frac{1}{(R_{2}-R_{1})^{p}}
       \int_{\tau_{2}}^{t}
       \sum_{x\in B_{R_{2}}(x_{0})}
       \ppos{\unk(x,\tau)-k}^{p+\theta-1}
       \dgw(x)
       \di\tau
       \,,
    \\
    B&=
       h^{p}
       (h-k)^{\theta-1}
       \int_{\tau_{2}}^{t}
       \msw(B_{R_{2}}(x_{0})\cap\{2h\ge \unk(x,\tau)>h\})
       \di\tau
       \,.
  \end{align*}
\end{lemma}

\begin{remark}
  \label{r:caccio}
  The term $A^{1/p}B^{(p-1)/p}$ in \eqref{eq:cacc_n} can be reduced to
  one containing only $A$ by means of Young's and Chebychev's
  inequalities.
\end{remark}

\begin{proof}
  We multiply \eqref{eq:pde} against
  $\zeta(x)^{p}\eta(t)^{p}\ppos{\unk(x,t)-h}^{\theta}$ and apply the
  well known formula of integration by parts
  \begin{multline*}
    \sum_{x,y\in V}
    \abs{\dif_{y}\unk(x)}^{p-2}
    \dif_{y}\unk(x)
    f(x)
    \wgt(x,y)
    \\=
    -
    \frac{1}{2}
    \sum_{x,y\in V}
    \abs{\dif_{y}\unk(x)}^{p-2}
    \dif_{y}\unk(x)
    \dif_{y}f(x)
    \wgt(x,y)
    \,,
  \end{multline*}
  where $f:V\to \R$ has finite support. Below we denote $B_{R}(x_{0})=B_{R}$ for simplicity of notation.

  We obtain
  \begin{multline}
    \label{eq:cacc_i}
    J_{1}+J_{2}
    :=
    \frac{1}{\theta+1}
    \sum_{x\in B_{R_{2}}}
    \ppos{\unk(x,t)-h}^{\theta+1}
    \zeta(x)^{p}
    \eta(t)^{p}
    \dgw(x)
    \\
    +
    \frac{1}{2}
    \int_{0}^{t}
    \sum_{x,y\in V}
    \abs{\dif_{y}\unk(x,\tau)}^{p-2}
    \dif_{y}\unk(x,\tau)
    \dif_{y}[
    \ppos{\unk(x,\tau)-h}^{\theta}
    \zeta(x)^{p}
    ]
    \wgt(x,y)
    \eta(\tau)^{p}
    \di\tau
    \\
    =
    \frac{p}{\theta+1}
    \int_{0}^{t}
    \sum_{x\in B_{R_{2}}}
    \ppos{\unk(x,\tau)-h}^{\theta+1}
    \zeta(x)^{p}
    \eta(\tau)^{p-1}
    \eta'(\tau)
    \dgw(x)
    \di\tau
    =:J_{3}
    \,.
  \end{multline}
  We split $J_{2}$ according to the equality 
  \begin{equation*}
    \dif_{y}[
    \ppos{\unk(x,\tau)-h}^{\theta}
    \zeta(x)^{p}
    ]
    =
    \zeta(y)^{p}
    \dif_{y}
    \ppos{\unk(x,\tau)-h}^{\theta}
    +
    \ppos{\unk(x,\tau)-h}^{\theta}
    \dif_{y} \zeta(x)^{p}
    \,.
  \end{equation*}
  Next we appeal to Lemma~\ref{l:monot} with $\unkii=0$ to get
  \begin{equation}
    \label{eq:cacc_iii}
    \abs{\dif_{y}\unk(x,\tau)}^{p-2}
    \dif_{y}\unk(x,\tau)
    \dif_{y}[
    \ppos{\unk(x,\tau)-h}^{\theta}
    ]
    \ge
    \gamma_{0}
    \Abs{
      \dif_{y}
      \ppos{\unk(x,\tau)-h}^{\frac{p+\theta-1}{p}}
    }^{p}
    \,.
  \end{equation}
  Thus from \eqref{eq:cacc_i} we infer the bound
  \begin{equation}
    \label{eq:cacc_iv}
    J_{1}
    +
    J_{21}
    +
    J_{22}
    \le
    J_{3}
    +
    J_{23}
    \,,
  \end{equation}
  where
  \begin{align*}
    J_{21}
    &=
      \gamma_{0}
      \int_{0}^{t}
      \sum_{x,y\in V}
      \Abs{
      \dif_{y}
      \ppos{\unk(x,\tau)-h}^{\frac{p+\theta-1}{p}}
      }^{p}
      \zeta(y)^{p}
      \wgt(x,y)
      \eta(\tau)^{p}
      \di\tau
      \,,
    \\
    J_{22}
    &=
      \frac{1}{4}
      \int_{0}^{t}
      \sum_{x,y\in V}
      \abs{\dif_{y}\unk(x,\tau)}^{p-2}
      \dif_{y}\unk(x,\tau)
      \dif_{y}[
      \ppos{\unk(x,\tau)-h}^{\theta}
      ]
      \zeta(y)^{p}
      \wgt(x,y)
      \eta(\tau)^{p}
      \di\tau
      \,,
    \\
    J_{23}
    &=
      \frac{1}{2}
      \int_{0}^{t}
      \sum_{x,y\in V}
      \abs{\dif_{y}\unk(x,\tau)}^{p-1}
      \abs{\dif_{y}\zeta(x)^{p}}
      \ppos{\unk(x,\tau)-h}^{\theta}
      \eta(\tau)^{p}
      \wgt(x,y)
      \di\tau
      \,.
  \end{align*}
  The reason to preserve the fraction $J_{22}$ of $J_{2}$ (rather than
  treating it as in $J_{21}$) will become apparent presently. Let us
  introduce the functions
  \begin{gather*}
    H(x,y;r)
    =
    \max[
    \ppos{\unk(x,\tau)-r}
    ,
    \ppos{\unk(y,\tau)-r}
    ]
    \,,
    \\
    \chi_{x,y}
    =
    1
    \,,
    \quad
    \text{if $H(x,y;h)>0$;}
    \qquad
    \chi_{x,y}
    =
    0
    \,,
    \quad
    \text{if $H(x,y;h)=0$.}
  \end{gather*}
  Note that $r>0$ is arbitrary in the definition of $H$ but we fix $r=h$ in the definition of $\chi_{x,y}$.
  Next we select $0<k<h$; by elementary calculations and Young's inequality we get
  \begin{equation*}
    \begin{split}
      J_{23}
      &\le
      \frac{p}{2}
      \int_{0}^{t}
      \sum_{x,y\in V}
      \abs{\dif_{y} \unk(x,\tau)}^{p-1}
      \abs{\dif_{y} \zeta(x)}
      (\zeta(x)+\zeta(y))^{p-1}
      H(x,y;k)^{\theta}
      \chi_{x,y}
      \wgt(x,y)
      \eta(\tau)^{p}
      \di\tau
      \\
      &\le
      \eps
      \int_{0}^{t}
      \sum_{x,y\in V}
      \abs{\dif_{y}\unk(x,\tau)}^{p}
      (\zeta(x)^{p}+\zeta(y)^{p})
      H(x,y;k)^{\theta-1}
      \chi_{x,y}
      \wgt(x,y)
      \eta(\tau)^{p}
      \di\tau
      \\
      &\quad+
      \gamma \eps^{1-p}
      \int_{0}^{t}
      \sum_{x,y\in V}
      \abs{\dif_{y}\zeta(x)}^{p}
      H(x,y;k)^{p+\theta-1}
      \wgt(x,y)
      \eta(\tau)^{p}
      \di\tau
      =:
      J_{231}
      +
      J_{232}
      \,.
    \end{split}
  \end{equation*}
  We want to absorb partially the term $J_{231}$ into $J_{22}$, for a
  suitable choice of $\eps$. To this end we observe that by a change
  of variables we have
  \begin{equation*}
    \begin{split}
      J_{22}
      &=
      \frac{1}{4}
      \int_{0}^{t}
      \sum_{x,y\in V}
      \abs{\dif_{x}\unk(y,\tau)}^{p-1}
      \abs{\dif_{x}\ppos{\unk(y,\tau)-h}^{\theta}}
      \zeta(x)^{p}
      \wgt(y,x)
      \eta(\tau)^{p}
      \di\tau
      \\
      &=
      \frac{1}{4}
      \int_{0}^{t}
      \sum_{x,y\in V}
      \abs{\dif_{y}\unk(x,\tau)}^{p-1}
      \abs{\dif_{y}\ppos{\unk(x,\tau)-h}^{\theta}}
      \zeta(x)^{p}
      \wgt(x,y)
      \eta(\tau)^{p}
      \di\tau
      \\
      &=
      \frac{1}{8}
      \int_{0}^{t}
      \sum_{x,y\in V}
      \abs{\dif_{y}\unk(x,\tau)}^{p-1}
      \abs{\dif_{y}\ppos{\unk(x,\tau)-h}^{\theta}}
      (\zeta(x)^{p}+\zeta(y)^{p})
      \wgt(x,y)
      \eta(\tau)^{p}
      \di\tau
      \,.
    \end{split}
  \end{equation*}
  Then by elementary calculus
  \begin{multline}
    \label{eq:cacc_j}
    \chi_{x,y}
    \abs{\dif_{y}\ppos{\unk(x,\tau)-h}^{\theta}}
    \ge
    \chi_{x,y}
    \theta
    \abs{\dif_{y}\ppos{\unk(x,\tau)-h}}
    H(x,y;h)^{\theta-1}
    \\
    \ge
    \chi_{x,y}
    \theta
    \abs{\dif_{y}\ppos{\unk(x,\tau)-h}}
    H(x,y;k)^{\theta-1}
    \,.
  \end{multline}
  Next we discriminate three cases in \eqref{eq:cacc_j}, aggregating
  equivalent symmetric cases: i) $\unk(x,\tau)>h$, $\unk(y,\tau)>h$. In this
  case clearly
  \begin{equation*}
    \abs{\dif_{y}\ppos{\unk(x,\tau)-h}}
    =
    \abs{\dif_{y}\unk(x,\tau)}
    \,.
  \end{equation*}
  ii) $\unk(x,\tau)>2h$, $h\ge \unk(y,\tau)$.
  Then
  \begin{equation*}
    \abs{\dif_{y}\ppos{\unk(x,\tau)-h}}
    \ge
    \frac{\unk(x,\tau)}{2}
    \ge
    \frac{1}{2}
    \abs{\dif_{y}\unk(x,\tau)}
    \,.
  \end{equation*}
  iii) $2h \ge \unk(x,\tau)>h\ge \unk(y,\tau)$. In this case $J_{22}$
  does not offer any help. We rather bound directly this part of $J_{231}$ as shown below.

  Collecting the estimates above we see that, provided $\eps\le 1/16$,
  \begin{multline*}
    J_{231}
    \le
    J_{22}
    +
    \eps
    2^{p+2}
    h^{p}
    (h-k)^{\theta-1}
    \int_{\tau_{2}}^{t}
    \msw(B_{R_{2}}\cap\{2h\ge \unk(x,\tau)>h\})
    \di\tau
    \,.
  \end{multline*}
  Hence we have transformed \eqref{eq:cacc_iv} into
  \begin{equation}
    \label{eq:cacc_jj}
    J_{1}+J_{21}
    \le
    J_{3}
    +
    \gamma \eps B
    +
    \gamma \eps^{1-p}
    A
    \,,
  \end{equation}
  where $A$ and $B$ are as in the statement.
  
  Finally we check whether the root $\eps$ of $\eps B=\eps^{1-p}A$ is
  less than $1/16$; on distinguishing the cases $\eps\le 1/16$,
  $\eps>1/16$ we get the inequality in \eqref{eq:cacc_n}. 
\end{proof}

\begin{lemma}
  \label{l:cacc2}
  Let $\unk\in L^{\infty}(0,T;\ell^{q}(V))$, for a given $q> 1$, be a solution of \eqref{eq:pde} in $V\times(0,T)$. Then
  for all $T>\tau_{1}>\tau_{2}>0$, $h\ge0$, we have for all $0<t<T$
  \begin{multline}
    \label{eq:cacc2_n}
    \sup_{\tau_{1}<\tau<t}
    \sum_{x\in V}
    \ppos{\unk(x,\tau)-h}^{q}
    \dgw(x)
    +
    \int_{\tau_{1}}^{t}
    \sum_{x,y\in V}
    \Abs{\dif_{y}\ppos{\unk(x,\tau)-h}^{\frac{p+q-2}{p}}}^{p}
    \wgt(x,y)
    \di \tau
    \\
    \le
    \frac{\gamma}{\tau_{1}-\tau_{2}}
    \int_{\tau_{2}}^{t}
    \sum_{x\in V}
    \ppos{\unk(x,\tau)-h}^{q}
    \dgw(x)
    \di\tau
    \,.
  \end{multline}
  We have also, if condition \eqref{eq:init} is satisfied,
  \begin{multline}
    \label{eq:cacc2_nn}
    \sup_{0<\tau<t}
    \sum_{x\in V}
    \ppos{\unk(x,\tau)-h}^{q}
    \dgw(x)
    +
    \int_{0}^{t}
    \sum_{x,y\in V}
    \Abs{\dif_{y}\ppos{\unk(x,\tau)-h}^{\frac{p+q-2}{p}}}^{p}
    \wgt(x,y)
    \di \tau
    \\
    \le
    \gamma
    \int_{0}^{t}
    \sum_{x\in V}
    \ppos{\unk_{0}(x)-h}^{q}
    \dgw(x)
    \di\tau
    \,.
  \end{multline}
\end{lemma}

\begin{proof}
  Let us prove \eqref{eq:cacc2_n}; the inequality \eqref{eq:cacc2_nn} is proved similarly.
  
  We multiply \eqref{eq:pde} against
  $\zeta(x)\eta(t)\ppos{\unk(x,t)-h}^{q-1}$; on integrating by parts as in
  the proof of Lemma~\ref{l:cacc} we obtain
  \begin{equation}
    \label{eq:cacc2_i}
    \begin{split}
      &\frac{1}{q}
      \sum_{x\in V}
      \zeta(x)
      \ppos{\unk(x,t)-h}^{q}
      \dgw(x)
      \eta(t)
      \\
      &\quad+
      \int_{0}^{t}
      \sum_{x,y\in V}
      \abs{\dif_{y}\unk(x,\tau)}^{p-2}
      \dif_{y}\unk(x,\tau)
      \zeta(y)
      \dif_{y}\ppos{\unk(x,\tau)-h}^{q-1}
      \wgt(x,y)
      \eta(\tau)
      \di\tau
      \\
      &\quad+
      \int_{0}^{t}
      \sum_{x,y\in V}
      \abs{\dif_{y}\unk(x,\tau)}^{p-2}
      \dif_{y}\unk(x,\tau)
      \dif_{y}\zeta(x)
      \ppos{\unk(x,\tau)-h}^{q-1}
      \wgt(x,y)
      \eta(\tau)
      \di\tau
      \\
      &\qquad=
      \frac{1}{q}
      \int_{0}^{t}
      \sum_{x\in V}
      \zeta(x)
      \ppos{\unk(x,\tau)-h}^{q}
      \dgw(x)
      \eta'(\tau)
      \di\tau
      \,.
    \end{split}
  \end{equation}
  We estimate next the second integral in \eqref{eq:cacc2_i}. The
  absolute value of the integrand is bounded from above by
  \begin{multline*}
    \frac{1}{R_{2}-R_{1}}
    \sum_{x,y\in B_{R_{2}+1}}
    \abs{\dif_{y}\unk(x,\tau)}^{p-1}
    \ppos{\unk(x,\tau)-h}^{q-1}
    \wgt(x,y)
    \\
    \le
    \frac{1}{R_{2}-R_{1}}
    \sum_{x,y\in B_{R_{2}+1}}
    \big(
    \unk(x,\tau)^{p+q-2}
    +
    \unk(y,\tau)^{p-1}
    \unk(x,\tau)^{q-1}
    \big)
    \wgt(x,y)
    \le
    \frac{C_{u}}{R_{2}-R_{1}}
    \,,
  \end{multline*}
  where $C_{u}$ is independent of $R_{i}$. Owing to $p+q-2>q$ and to Remark~\ref{r:lp_scale}, to this end it is only left to observe that
  \begin{multline*}
    \sum_{x,y\in B_{R_{2}+1}}
    \unk(y,\tau)^{p-1}
    \unk(x,\tau)^{q-1}
    \wgt(x,y)
    \\
    \le
    \Big(
    \sum_{y\in V}
    \unk(y,\tau)^{(p-1)q}
    \dgw(y)
    \Big)^{\frac{1}{q}}
    \Big(
    \sum_{x\in V}
    \unk(x,\tau)^{q}
    \dgw(x)
    \Big)^{\frac{q-1}{q}}
    \,,
  \end{multline*}
  and to use once more Remark~\ref{r:lp_scale}, since $(p-1)q>q$.

  The sought after estimates follows immediately upon applying
  Lemma~\ref{l:monot} with $\unkii=0$ and then letting first
  $R_{2}\to\infty$ and then $R_{1}\to\infty$. 
\end{proof}

\begin{remark}
  \label{r:prelim_diff}
  Lemma~\ref{l:cacc2} is still in force if $\unk$ is the difference of
  two solutions to \eqref{eq:pde}. The proof is the same, when we
  start from the difference of the two equations and recall
  Lemma~\ref{l:monot}.
\end{remark}

\subsection{Existence and comparison}
\label{s:prelim_ex}

\begin{proposition}
  \label{p:ex}
  Let $\unk_{0}\in\ell^{q}(V)$, $q>1$. Then
  \eqref{eq:pde}--\eqref{eq:init} has a solution in
  $L^{\infty}(0,+\infty;\ell^{q}(V))$. If $\unk_{0}\ge 0$ then $\unk\ge 0$.
\end{proposition}

\begin{proof}
  Let $\unk_{0}\in\ell^{q}(V)$, $q>1$. Define for $n\ge 1$ $\unk_{n}$ as the solution to
  \begin{alignat}{2}
    \label{eq:exist_pde_n}
    \pder{\unk_{n}}{t}(x,t)
    &=
    \Lapl_{p}\unk_{n}(x,t)
    \,,
    &\qquad&
    x\in B_{n}
    \,,
    t>0
    \,,
    \\
    \label{eq:exist_init_n}
    \unk_{n}(x,0)
    &=
    \unk_{0}(x)
    \,,
    &\qquad&
    x\in B_{n}
    \,,
    \\
    \label{eq:exist_dir_n}
    \unk_{n}(x,t)
    &=
    0
    \,,
    &\qquad&
    x\not\in B_{n}
    \,,
    t\ge 0
    \,.
  \end{alignat}
  In practice this is a finite system of ordinary differential
  equations, uniquely solvable in the class $C^{1}(0,T)$ at least as long as the solution stays
  bounded over $(0,T)$.

  In this connection, we rewrite \eqref{eq:exist_pde_n}, \eqref{eq:exist_dir_n} as
  \begin{equation*}
    \unk_{n}(x,t)^{q-1}
    \pder{\unk_{n}}{t}(x,t)
    =
    \unk_{n}(x,t)^{q-1}
    \Lapl_{p}\unk_{n}(x,t)
    \,,
    \qquad
    x\in V
    \,,
    t>0
    \,,
  \end{equation*}
  where we stress that the equality holds for all $x\in V$. In this
  Subsection we denote $s^{q-1}=\abs{s}^{q-1}\textup{sign}(s)$ for all
  $s\in \R$. Thus, summing over $x\in V$ and integrating by parts both
  in $t$ and in $x$ (in the suitable sense) we see that the elliptic
  part of the equation yields a nonnegative contribution, so that
  \begin{equation}
    \label{eq:exist_energy_n}
    \sum_{x\in V}
    \abs{\unk_{n}(x,t)}^{q}
    \dgw(x)
    \le
    \sum_{x\in B_{n}}
    \abs{\unk_{0}(x)}^{q}
    \dgw(x)
    \le
    \norma{\unk_{0}}{\ell^{q}(V)}^{q}
    \,.
  \end{equation}
  In turn, as explained in Remark~\ref{r:lp_scale}, this implies stable
  sup bounds for $\unk_{n}$ which, together with the discrete character
  of the $p$-laplacian and with the equation \eqref{eq:exist_pde_n},
  also give stable sup bounds for the time derivative
  $\partial \unk_{n}/\partial t$, for each fixed $x$. However $V$ is
  countable, so that this is enough to enable us to extract a
  subsequence, still denoted by $\unk_{n}$ such that
  \begin{equation}
    \label{eq:exist_unif_conv}
    \unk_{n}(x,t)
    \to
    \unk(x,t)
    \,,
    \quad
    \pder{\unk_{n}}{t}(x,t)
    \to
    \pder{\unk}{t}(x,t)
  \end{equation}
  for each $x\in V$, uniformly for $t\in [0,T]$, where we have made use
  of the equation again to obtain convergence for the time
  derivative. Finally owing to \eqref{eq:exist_energy_n} we have
  \begin{equation}
    \label{eq:exist_energy_nj}
    \sum_{x\in V}
    \abs{\unk(x,t)}^{q}
    \dgw(x)
    \le
    \norma{\unk_{0}}{\ell^{q}(V)}^{q}
    \,,
    \qquad
    t>0
    \,.
  \end{equation}
  It is easily seen that $\unk\in L^{\infty}(0,+\infty;\ell^{q}(V))$
  is a solution to \eqref{eq:pde}--\eqref{eq:init}. If
  $\unk_{0}\ge 0$, we appeal to our next result to prove that
  $\unk\ge 0$.
\end{proof}

\begin{proposition}[Comparison]
  \label{p:compare}
  If $\unk_{1}$, $\unk_{2}\in L^{\infty}(0,T; \ell^{q}(V))$ solve
  \eqref{eq:pde}--\eqref{eq:init} with $\unk_{01}$, $\unk_{02}\in\ell^{q}(V)$, $\unk_{01}\ge \unk_{02}$, then
  $\unk_{1}\ge\unk_{2}$.
\end{proposition}

\begin{proof}
  According to Remark~\ref{r:lp_scale} and to Definition~\ref{d:sol},
  we may assume $q>1$.  Define $\unkiii=\unk_{2}-\unk_{1}$. Then
  $\unkiii$ does not solve \eqref{eq:pde}, but we may still apply
  \eqref{eq:cacc2_nn} (with $h=0$) to it, see
  Remark~\ref{r:prelim_diff}. This proves $\ppos{\unkiii}=0$ and thus
  the statement.
\end{proof}

\subsection{Elementary inequalities}
\label{s:prelim_elem}
We record for future use two immediate consequences of \eqref{eq:fkf_nd}, \eqref{eq:fkf_ni}:
\begin{alignat}2
  \label{eq:fkf_above}
  \fkf(sa)^{-1}
  &\le
  s^{\omega}
  \fkf(a)^{-1}
  \,,
  &\qquad&
  s\ge 1
  \,,
  a>0
  \,;
  \\
  \label{eq:fkf_below}
  \fkf(\sigma a)^{-1}
  &\le
  \sigma^{\frac{p}{\SpDim}}
  \fkf(a)^{-1}
  \,,
  &\qquad&
  0<\sigma\le 1
  \,,
  a>0
  \,.
\end{alignat}
Also the following Lemma relies on \eqref{eq:fkf_nd} and will be used in a
context where it is important that $\nu<1/(p-1)$.
\begin{lemma}
  \label{l:dcf}
  Let $\nu=\SpDim(p-2)/[(\SpDim(p-2)+p)(p-1)]$ and $b>0$.  Then the function
  \begin{equation*}
    \tau\mapsto \tau^{\nu}\dcf_{1}^{(-1)}(\tau^{-1}b)^{\frac{p-2}{p-1}}
    \,,
    \qquad
    \tau>0
    \,,
  \end{equation*}
  is nondecreasing.
\end{lemma}
\begin{proof}
  Equivalently we show that
  \begin{equation*}
    r\mapsto r^{-\alpha}\dcf_{1}^{(-1)}(r)^{p-2}
  \end{equation*}
  is nonincreasing for $\alpha=\nu(p-1)$. Set $s=\dcf_{1}^{(-1)}(r)$, so that by definition of $\dcf_{1}$
  \begin{equation*}
    r^{-\alpha}\dcf_{1}^{(-1)}(r)^{p-2}
    =
    s^{(1-\alpha)(p-2)}\fkf(s^{-1})^{-\alpha}
    =
    [s^{-\frac{p}{\SpDim}}\fkf(s^{-1})]^{-\alpha}
    \,.
  \end{equation*}
  By assumption \eqref{eq:fkf_nd} the latter quantity is indeed
  nonincreasing in $s$ which however is a nondecreasing function of
  $r$.
\end{proof}

\begin{lemma}
  \label{l:bbl}
  Under assumption \eqref{eq:fkf_bbl} we have that if $R$, $s>0$, $c\ge 1$ and
  \begin{equation}
    \label{eq:bbl_m}
    R^{p}
    =
    cs
    \dcf_{1}^{(-1)}(s^{-1})^{p-2}
    \,,
  \end{equation}
  then
  \begin{equation}
    \label{eq:bbl_mm}
    \msw(B_{\floor{R}})
    \le
    \gamma(c)
    \psi_{1}^{(-1)}
    (s^{-1})^{-1}
    \,.
  \end{equation}
\end{lemma}

\begin{proof}
  Let $\tau>0$ be such that $s^{-1}=\dcf_{1}(\tau)=\tau^{p-2}\fkf(\tau^{-1})$. Then
  \begin{equation*}
    c^{-1}R^{p}
    =
    \fkf(\tau^{-1})^{-1}
    \,.
  \end{equation*}
  On the other hand, on setting $v=\msw(B_{\floor{R}})$ and invoking \eqref{eq:fkf_bbl} we get
  \begin{equation*}
    c^{-1}R^{p}
    \ge
    c^{-1}\rdf(v)^{p}
    \ge
    c^{-1}
    \gamma_{0}
    \fkf(v)^{-1}
    \ge
    \fkf((\gamma_{0}c^{-1})^{\frac{\SpDim}{p}}v)^{-1}
    \,,
  \end{equation*}
  where we also used \eqref{eq:fkf_below}. Since $\fkf$ is nonincreasing, the result follows.
\end{proof}

\section{Proofs of Proposition~\ref{p:linf_meas} and Corollary~\ref{co:linf_int}}
\label{s:linf}

\begin{proof}[Proof of Proposition~\ref{p:linf_meas}]
  By assumption, and by Remark~\ref{r:lp_scale},
  $\unk\in L^{\infty}(0,T;\ell^{q}(V))$ for some $q> 1$; then for all
  $k>0$ the cut function $\ppos{\unk(t)-k}$ is finitely supported. For
  given $0<\sigma_{1}<\sigma_{2}<1/2$, $k>0$, $0<t<T$ define the
  decreasing sequences
  \begin{gather*}
    k_{i}
    =
    k[
    1-\sigma_{2}
    +
    2^{-i}
    (\sigma_{2}-\sigma_{1})
    ]
    \,,
    \qquad
    i=0\,,1\,,2\,,\dots
    \\
    t_{i}
    =
    \frac{t}{2}
    [
    1-\sigma_{2}
    +
    2^{-i}
    (\sigma_{2}-\sigma_{1})
    ]
    \,,
    \qquad
    i=0\,,1\,,2\,,\dots
  \end{gather*}
  and let $f_{i}(x,\tau)=\ppos{\unk(x,\tau)-k_{i}}^{\nu}$, where $\nu=(p+q-2)/p$. Let also 
  \begin{align*}
    \mcut_{i}(\tau)
    &=
      \msw(
      \{
      x\in V
      \mid
      \unk(x,\tau)>k_{i}
      \}
      )
      \,,
      \quad
      \Mcut_{i}
      =
      \sup_{t_{i}<\tau<t}
      \mcut_{i}(\tau)
      \,,
    \\
    \gnor_{i}(\tau)
    &=
      \sum_{x,y\in V}
      \abs{\dif_{y}f_{i}(x,\tau)}^{p}
      \wgt(x,y)
      \,.
  \end{align*}
  Since $b:=q/\nu<p$, it follows from Faber-Krahn inequality
  \eqref{eq:fk} and H\"older's and Young's inequalities that
  \begin{multline}
    \label{eq:fkb}
    \sum_{x\in V}
    f_{i+1}(x,\tau)^{b}
    \dgw(x)
    \le
    \mcut_{i+1}(\tau)^{1-\frac{b}{p}}
    \fkf(\mcut_{i+1}(\tau))^{-\frac{b}{p}}
    \gnor_{i+1}(\tau)^{\frac{b}{p}}
    \\
    \le
    \eps^{\frac{p}{b}}
    \gnor_{i+1}(\tau)
    +
    \eps^{-\frac{p}{p-b}}
    \fkf(\mcut_{i+1}(\tau))^{-\frac{b}{p-b}}
    \mcut_{i+1}(\tau)
    \,.
  \end{multline}
  Here $\eps>0$ is arbitrary and will be selected below. 

  We integrate \eqref{eq:fkb} over $(t_{i+1},t)$ to find
  \begin{equation}
    \label{eq:linf_i}
    \begin{split}
      &\int_{t_{i+1}}^{t}
      \sum_{x\in V}
      f_{i+1}(x,\tau)^{b}
      \dgw(x)
      \di\tau
      \le
      \eps^{\frac{p}{b}}
      \int_{t_{i+1}}^{t}
      \gnor_{i+1}(\tau)
      \di\tau
      \\
      &\qquad+
      \eps^{-\frac{p}{p-b}}
      \int_{t_{i+1}}^{t}
      \fkf(\mcut_{i+1}(\tau))^{-\frac{b}{p-b}}
      \mcut_{i+1}(\tau)
      \di\tau
      \\
      &\quad\le
      \eps^{\frac{p}{b}}
      \int_{t_{i+1}}^{t}
      \gnor_{i+1}(\tau)
      \di\tau
      +
      \eps^{-\frac{p}{p-b}}
      t
      \fkf(\Mcut_{i+1})^{-\frac{b}{p-b}}
      \Mcut_{i+1}
      \,.
    \end{split}
  \end{equation}
  Next we invoke Lemma~\ref{l:cacc2} with $\tau_{1}=t_{i}$, $\tau_{2}=t_{i+1}$, $h=k_{i}$, to infer
  \begin{equation}
    \label{eq:linf_ii}
    \begin{split}
      &
      L_{i}:=
      \sup_{t_{i}<\tau<t}
      \sum_{x\in V}
      f_{i}(x,\tau)^{b}
      \dgw(x)
      +
      \int_{t_{i}}^{t}
      \gnor_{i}(\tau)
      \di\tau
      \\
      &\quad
      \le
      \frac{\gamma 2^{i}}{t(\sigma_{2}-\sigma_{1})}
      \int_{t_{i+1}}^{t}
      \sum_{x\in V}
      f_{i+1}(x,\tau)^{b}
      \dgw(x)
      \di\tau
      \\
      &\quad\le
      \frac{\gamma 2^{i}}{t(\sigma_{2}-\sigma_{1})}
      \eps^{\frac{p}{b}}
      \int_{t_{i+1}}^{t}
      \gnor_{i+1}(\tau)
      \di\tau
      \\
      &\qquad
      +
      \frac{\gamma 2^{i}}{\sigma_{2}-\sigma_{1}}
      \eps^{-\frac{p}{p-b}}
      \fkf(\Mcut_{i+1})^{-\frac{b}{p-b}}
      \Mcut_{i+1}
      \,,
    \end{split}
  \end{equation}
  where the second inequality follows of course from \eqref{eq:linf_i}. For a $\delta>0$ to be chosen, select ($\gamma$ denotes here the constant in \eqref{eq:linf_ii})
  \begin{equation*}
    \frac{\gamma 2^{i}}{t(\sigma_{2}-\sigma_{1})}
    \eps^{\frac{p}{b}}
    =
    \delta
    \qquad
    \text{i.e.,}
    \qquad
    \eps
    =
    \gamma_{0}
    \delta^{\frac{b}{p}}
    t^{\frac{b}{p}}
    (\sigma_{2}-\sigma_{1})^{\frac{b}{p}}
    2^{-\frac{b}{p}i}
    \,.
  \end{equation*}
  On substituting this choice of $\eps$ in \eqref{eq:linf_ii} we arrive
  at an estimate which can be successfully iterated, that is
  \begin{equation}
    \label{eq:linf_iii}
    L_{i}
    \le
    \delta L_{i+1}
    +
    \frac{\gamma 2^{\frac{pi}{p-b}}}{(\sigma_{2}-\sigma_{1})^{\frac{p}{p-b}}}
    \delta^{-\frac{b}{p-b}}
    t^{-\frac{b}{p-b}}
    \fkf(\Mcut_{\infty})^{-\frac{b}{p-b}}
    \Mcut_{\infty}
    \,.
  \end{equation}
  Here we set
  \begin{gather}
    \label{eq:linf_ij}
    t_{\infty}
    =
    \lim_{i\to\infty}
    t_{i}
    =
    \frac{t}{2}(1-\sigma_{2})
    \,,
    \qquad
    k_{\infty}
    =
    \lim_{i\to\infty} k_{i}
    =
    k(1-\sigma_{2})
    \,,
    \\
    \label{eq:linf_ijj}
    M_{\infty}
    =
    \sup_{t_{\infty}<\tau<t}
    \msw(
    \{
    x\in V
    \mid
    \unk(x,\tau)>k_{\infty}
    \}
    )
    \,.
  \end{gather}
  On iterating \eqref{eq:linf_iii} we infer
  \begin{equation*}
    L_{0}
    \le
    \delta^{j}
    L_{j}
    +
    \Big(
    \sum_{i=0}^{j}
    \delta^{i}
    2^{\frac{pi}{p-b}}
    \Big)
    \frac{\gamma}{(\sigma_{2}-\sigma_{1})^{\frac{p}{p-b}}}
    t^{-\frac{b}{p-b}}
    \fkf(\Mcut_{\infty})^{-\frac{b}{p-b}}
    \Mcut_{\infty}
    \,,
  \end{equation*}
  which yields as $j\to\infty$, provided we select  $\delta<2^{-p/(p-b)}$,
  \begin{equation}
    \label{eq:linf_j}
    \begin{split}
      &
      \sup_{t(1-\sigma_{1})/2<\tau<t}
      \sum_{x\in V}
      \ppos{\unk(x,\tau)-k(1-\sigma_{1})}^{q}
      \dgw(x)
      \le L_{0}
      \\
      &\quad
      \le
      \frac{\gamma}{(\sigma_{2}-\sigma_{1})^{\frac{q}{p-2}}}
      t^{-\frac{q}{p-2}}
      \fkf(\Mcut_{\infty})^{-\frac{q}{p-2}}
      \Mcut_{\infty}
      \,,
    \end{split}
  \end{equation}
  for $\Mcut_{\infty}$ as in \eqref{eq:linf_ijj}, owing also to $b/(p-b)=q/(p-2)$.

  The proof will be concluded by a second process of iteration, built on \eqref{eq:linf_j}.
  Let $1/2>\sigma>0$ and $k>0$, and define the increasing sequences
  \begin{gather*}
    \tau_{n}
    =
    \frac{t}{2}(1-\sigma 2^{-n})
    \,,
    \qquad
    h_{n}
    =
    k(1-\sigma 2^{-n})
    \,,
    \\
    \bar h_{n}
    =
    \frac{h_{n}+h_{n+1}}{2}
    =
    k(1-3\sigma 2^{-n-2})
    \,,
    \qquad
    n\ge 0
    \,,
  \end{gather*}
  as well as the decreasing one
  \begin{equation*}
    Y_{n}
    =
    \sup_{\tau_{n}<\tau<t}
    \msw(\{x\in V\mid \unk(x,\tau)> h_{n}\})
    \,.
  \end{equation*}
  Next we apply Chebychev's inequality to find
  \begin{equation}
    \label{eq:linf_jj}
    Y_{n+1}
    \le
    2^{(n+2)q}
    \sigma^{-q}
    k^{-q}
    \sup_{\tau_{n+1}<\tau<t}
    \sum_{x\in V}
    \ppos{\unk(x,\tau)-\bar h_{n}}
    \dgw(x)
    \,.
  \end{equation}
  The right hand side of \eqref{eq:linf_jj} is then majorized by appealing to \eqref{eq:linf_j}, where we select
  \begin{equation*}
    \sigma_{1}
    =
    3\sigma
    2^{-n-2}
    \,,
    \quad
    \sigma_{2}
    =
    \sigma
    2^{-n}
    \,,
  \end{equation*}
  in order to obtain
  \begin{equation}
    \label{eq:linf_jjj}
    Y_{n+1}
    \le
    \gamma
    \sigma^{-\frac{q(p-1)}{p-2}}
    2^{\frac{n(p-2+q)}{p-2}}
    t^{-\frac{q}{p-2}}
    k^{-q}
    \fkf(Y_{n})^{-\frac{q}{p-2}}
    Y_{n}
    \,.
  \end{equation}
  In turn, on invoking our assumption \eqref{eq:fkf_nd}, we transform
  \eqref{eq:linf_jjj} into
  \begin{equation}
    \label{eq:linf_k}
    Y_{n+1}
    \le
    \gamma
    \sigma^{-\frac{q(p-1)}{p-2}}
    2^{\frac{n(p-2+q)}{p-2}}
    t^{-\frac{q}{p-2}}
    k^{-q}
    \fkf(Y_{0})^{-\frac{q}{p-2}}
    Y_{0}^{-\frac{p}{\SpDim}\,\frac{q}{p-2}}
    Y_{n}^{1+\frac{p}{\SpDim}\,\frac{q}{p-2}}
    \,.
  \end{equation}
  This inequality yields $Y_{n}\to0$ as $n\to\infty$ provided we choose $k$ so that (see \cite[Lemma~5.6 Ch.~II]{LSU})
  \begin{equation}
    \label{eq:linf_kk}
    k^{-1}
    t^{-\frac{1}{p-2}}
    \fkf(Y_{0})^{-\frac{1}{p-2}}
    \le
    \gamma_{0}(q,p,\SpDim)
    \,.
  \end{equation}
  In this connection we may assume e.g., $\sigma=1/4$. The proof is concluded when we remark that $Y_{n}\to0$ immediately implies
  \begin{equation*}
    \unk(x,t)
    \le
    k
    \,,
    \qquad
    x\in V
    \,.
  \end{equation*}
\end{proof}

\begin{remark}
  \label{r:cacc}
  We note that the proof of Proposition~\ref{p:linf_meas} makes use of
  the differential equation \eqref{eq:pde} only thru inequality
  \eqref{eq:cacc2_n}. This fact will be used below.
\end{remark}
\begin{proof}[Proof of Corollary~\ref{co:linf_int}]
  We remark on using Chebychev's inequality once more that in \eqref{eq:linf_kk}
  \begin{equation*}
    Y_{0}
    \le
    2^{r}k^{-r}
    \sup_{\frac{t}{4}<\tau<t}
    \sum_{x\in V}
    \unk(x,\tau)^{r}
    \dgw(x)
    \,.
  \end{equation*}
  Let us set
  \begin{equation*}
    E_{r}
    =
    \sup_{0<\tau<t}
    \sum_{x\in V}
    \unk(x,\tau)^{r}
    \dgw(x)
    \,.
  \end{equation*}
  Then \eqref{eq:linf_kk} is certainly fulfilled if
  \begin{equation}
    \label{eq:linf_int_i}
    k^{-1}
    t^{-\frac{1}{p-2}}
    \fkf(
    k^{-r}
    E_{r}
    )^{-\frac{1}{p-2}}
    =
    \gamma_{0}
    \,,
  \end{equation}
  where we also used \eqref{eq:fkf_above}.
  On the other hand, if we set
  \begin{equation*}
    \dcf_{r}(s)
    =
    s^{\frac{p-2}{r}}
    \fkf(s^{-1})
    \,,
    \qquad
    s>0
    \,,
  \end{equation*}
  then \eqref{eq:linf_int_i} amounts to
  \begin{equation}
    \label{eq:linf_int_ii}
    k
    =
    E_{r}^{\frac{1}{r}}
    \big[\dcf_{r}^{(-1)}
    \big(\gamma t^{-1}E_{r}^{-\frac{p-2}{r}}\big)
    \big]^{\frac{1}{r}}
    \le
    \gamma
    E_{r}^{\frac{1}{r}}
    \big[\dcf_{r}^{(-1)}
    \big(t^{-1}E_{r}^{-\frac{p-2}{r}}\big)
    \big]^{\frac{1}{r}}
    \,,
  \end{equation}
  where we have made use of the definition of $\dcf_{r}$ and of \eqref{eq:fkf_below}.
\end{proof}

\section{Proof of Theorem~\ref{t:l1}}
\label{s:l1}

Let $\unk_{0}\in\ell^{1}(V)$, $\unk_{0}\ge 0$. Then we have also
$\unk_{0}\in\ell^{2}(V)$ as noted in the Introduction, and we may
consider the solution $\unk\ge 0$ constructed according to
Subsection~\ref{s:prelim_ex}. First we bound the $\ell^{1}(V)$ norm of
the solution from above. We multiply the equation against
$\varTheta(\unk(x,\tau))\zeta(x)$ where $\zeta$ is as in
Section~\ref{s:prelim},
\begin{equation*}
  \varTheta(\unk)
  =
  \frac{
    \ppos{\unk-h}
  }{
    \unk +\eps
  }
  \,,
\end{equation*}
for any given $h>0$, $\eps>0$, and integrate by parts. The purpose of
the cut at level $h$ is to ease technically the argument. Since
$\varTheta$ is a nondecreasing function, reasoning as in the proof of
Lemma~\ref{l:cacc2} we easily obtain
\begin{multline*}
  \sum_{x\in V}
  \int_{0}^{\unk(x,t)}
  \frac{
    \ppos{s-h}
  }{
    s +\eps
  }
  \zeta(x)
  \dgw(x)
  \le
  \sum_{x\in V}
  \int_{0}^{\unk_{0}(x)}
  \frac{
    \ppos{s-h}
  }{
    s +\eps
  }
  \dgw(x)
  +
  K_{1}
  \\
  \le
  \sum_{x\in V}
  \unk_{0}(x)
  \dgw(x)
  +
  K_{1}
  \,,
\end{multline*}
where
\begin{multline*}
  K_{1}
  =
  \frac{1}{R_{2}-R_{1}}
  \int_{0}^{t}
  \sum_{x\in V}
  \abs{\dif_{y}\unk(x,\tau)}^{p-1}
  \chi_{\{\unk(\tau)>h\}}(x)
  \wgt(x,y)
  \di\tau
  \\
  \le
  \frac{h^{-1}}{R_{2}-R_{1}}
  \int_{0}^{t}
  \sum_{x\in V}
  \abs{\dif_{y}\unk(x,\tau)}^{p-1}
  \unk(x,\tau)
  \wgt(x,y)
  \di\tau
  \,.
\end{multline*}
Then we may proceed as in the proof of Lemma~\ref{l:cacc2} with $q=2$
and let $R_{2}\to \infty$ and then $R_{1}\to \infty$ to make $K_{1}$
vanish. Finally we let first $\eps\to 0$ and then $h\to 0$: on invoking the
monotone convergence theorem we get
\begin{equation}
  \label{eq:l1_above}
  \norma{\unk(t)}{\ell^{1}(V)}
  \le
  \norma{\unk_{0}}{\ell^{1}(V)}
  \,.
\end{equation}
Therefore from Corollary~\ref{co:linf_int} and Remark~\ref{r:dcf} we
infer that \eqref{eq:l1_nn} is satisfied.

In order to prove \eqref{eq:l1_n} we proceed as follows. We multiply
the equation against $\zeta(x)$ as above and integrate by parts
obtaining
\begin{equation}
  \label{eq:l1_j}
  \sum_{x\in V}
  \unk(x,t)
  \zeta(x)
  \dgw(x)
  +
  K_{2}
  =
  \sum_{x\in V}
  \unk_{0}(x)
  \zeta(x)
  \dgw(x)
  \,,
\end{equation}
where
\begin{equation}
  \label{eq:l1_entropy}
  \begin{split}
    \abs{K_{2}}
    &=
    \Abs{
      \int_{0}^{t}
      \sum_{x\in V}
      \abs{\dif_{y}\unk(x,\tau)}^{p-2}
      \dif_{y}\unk(x,\tau)
      \dif_{y}\zeta(x)
      \wgt(x,y)
      \di\tau
    }
    \\
    &\le
    \frac{1}{R_{2}-R_{1}}
    \int_{0}^{t}
    \sum_{x\in B_{R_{2}}+1}
    \abs{\dif_{y}\unk(x,\tau)}^{p-1}
    \wgt(x,y)
    \di\tau
    \\
    &\le
    \frac{2t}{(R_{2}-R_{1})\fkf^{(-1)}(2)^{p-2}}
    \norma{\unk_{0}}{\ell^{1}(V)}^{p-1}
    \,.
  \end{split}
\end{equation}
Here we reasoned as in \eqref{eq:lp_linf} (with $q=1$), exploiting $p>2$ and the
already proved bound \eqref{eq:l1_above}. Then we rewrite \eqref{eq:l1_j} as
\begin{equation*}
  \norma{\unk(t)}{\ell^{1}(V)}
  +
  K_{2}
  \ge
  \sum_{x\in B_{R_{1}}}
  \unk_{0}(x)
  \dgw(x)
  \,,
\end{equation*}
and let first $R_{2}\to \infty$ then $R_{1}\to\infty$ to obtain the
converse to \eqref{eq:l1_above}.

Finally we prove the entropy estimate \eqref{eq:l1_nnn}. First we invoke H\"older's inequality to bound
\begin{equation}
  \label{eq:l1_k}
  \begin{split}
    &I:=
    \int_{0}^{t}
    \sum_{x,y\in V}
    \abs{\dif_{y}\unk(x,\tau)}^{p-1}
    \wgt(x,y)
    \di\tau
    \\
    &\quad\le
    \Big(
    \int_{0}^{t}
    \sum_{x,y\in V}
    \tau^{-\delta(p-1)}
    (\unk(x,\tau)+\unk(y,\tau))^{(2-\theta)(p-1)}
    \wgt(x,y)
    \di\tau
    \Big)^{\frac{1}{p}}
    \\
    &\qquad\times
    \Big(
    \int_{0}^{t}
    \sum_{x,y\in V}
    \tau^{\delta}
    \abs{\dif_{y}\unk(x,\tau)}^{p}
    (\unk(x,\tau)+\unk(y,\tau))^{\theta-2}
    \wgt(x,y)
    \di\tau
    \Big)^{\frac{p-1}{p}}
    \\
    &\quad=:
    K_{3}^{\frac{1}{p}}
    K_{4}^{\frac{p-1}{p}}
    \,.
  \end{split}
\end{equation}
Here $\delta>0$ is to be chosen and we select
\begin{equation*}
  \theta
  =
  \frac{2p-3}{p-1}
  \in(1,2)
  \,,
  \quad
  \text{so that}
  \quad
  (2-\theta)(p-1)
  =
  1
  \,.
\end{equation*}
Thus 
\begin{equation}
  \label{eq:l1_kk}
  K_{3}
  \le
  2
  \int_{0}^{t}
  \tau^{-\delta(p-1)}
  \norma{\unk(\tau)}{\ell^{1}(V)}
  \di\tau
  \le
  \gamma
  \norma{\unk_{0}}{\ell^{1}(V)}
  t^{1- \delta(p-1)}
  \,,
\end{equation}
provided
\begin{equation}
  \label{eq:l1_kkk}
  \delta(p-1)<1
  \,.
\end{equation}
In order to bound $K_{4}$ we multiply the differential equation
against $\tau^{\delta}\unk^{\theta-1}$ and integrate by parts. After
dropping a positive contribution from the left hand side of the
resulting equality we obtain
\begin{equation}
  \label{eq:l1_kj}
  \begin{split}
    K_{4}
    &\le
    \gamma
    \int_{0}^{t}
    \sum_{x\in V}
    \tau^{\delta -1}
    \unk(x,\tau)^{\frac{p-2}{p-1}+1}
    \dgw(x)
    \di\tau
    \\
    &\le
    \gamma
    \norma{\unk_{0}}{\ell^{1}(V)}
    \int_{0}^{t}
    \tau^{\delta-1}
    \norma{\unk(\tau)}{\ell^{\infty}(V)}^{\frac{p-2}{p-1}}
    \di\tau
    \\
    &\le
    \gamma
    \norma{\unk_{0}}{\ell^{1}(V)}^{\frac{p-2}{p-1}+1}
    \int_{0}^{t}
    \tau^{\delta-1}
    \dcf_{1}^{(-1)}(\tau^{-1}\norma{\unk_{0}}{\ell^{1}(V)}^{2-p})^{\frac{p-2}{p-1}}
    \di\tau
    \,.
  \end{split}
\end{equation}
Select now $\nu<\delta<1/(p-1)$, where $\nu$ is the constant defined
in Lemma~\ref{l:dcf}. Accordingly, the last integral above is bounded
by
\begin{multline}
  \label{eq:l1_kjj}
  \int_{0}^{t}
  \tau^{\delta-\nu-1}
  \tau^{\nu}
  \dcf_{1}^{(-1)}(\tau^{-1}\norma{\unk_{0}}{\ell^{1}(V)}^{2-p})^{\frac{p-2}{p-1}}
  \di\tau
  \\
  \le
  t^{\nu}
  \dcf_{1}^{(-1)}(t^{-1}\norma{\unk_{0}}{\ell^{1}(V)}^{2-p})^{\frac{p-2}{p-1}}
  (\delta-\nu)^{-1}
  t^{\delta-\nu}
  \,.
\end{multline}
Collecting all the estimates in \eqref{eq:l1_k}--\eqref{eq:l1_kjj}, we finally arrive at
\begin{equation}
  \label{eq:l1_kkj}
  I
  \le
  \gamma
  t^{\frac{1}{p}}
  \norma{\unk_{0}}{\ell^{1}(V)}^{\frac{2(p-1)}{p}}
  \dcf_{1}^{(-1)}(t^{-1}\norma{\unk_{0}}{\ell^{1}(V)}^{2-p})^{\frac{p-2}{p}}
  \,.
\end{equation}

\section{Proof of Theorem~\ref{p:bbl} and Corollary~\ref{p:bbl2}}
\label{s:bbl}

\begin{proof}[Proof of Theorem~\ref{p:bbl}]
  Let $\unk$ be as in the statement of Theorem~\ref{p:bbl}. For all $t>0$, $R\in\N$ we write
  \begin{equation*}
    \norma{\unk(t)}{\ell^{1}(V)}
    =
    \norma{\unk(t)}{\ell^{1}(B_{R})}
    +
    \norma{\unk(t)}{\ell^{1}(V\setminus B_{R})}
    \,.
  \end{equation*}
  Here we denote for a fixed $x_{0}\in V$
  \begin{equation*}
    B_{R}
    =
    B_{R}(x_{0})
    \,,
    \quad
    \abs{x}
    =
    d(x,x_{0})
    \,,
    \quad
    x\in V
    \,.
  \end{equation*}
  For the sake of clarity let us denote  by $\zeta_{R_{1},R_{2}}$ the cutoff
  function defined in Section~\ref{s:prelim}. Let $\rho>4R$, $R\ge 4$,
  $\rho$, $R\in \N$ and $\phi=1-\zeta_{R,2R}$. We
  use $\abs{x}^{\alpha}\phi(x)\zeta_{\rho,2\rho}(x)$ as a testing function in \eqref{eq:pde},
  for a fixed $0<\alpha<1$. We obtain, assuming in addition that $R$ is
  so large as $\unk_{0}(x)=0$ for $x\not\in B_{R}$,
  \begin{multline*}
    \sum_{x\in V}
    \abs{x}^{\alpha}
    \phi(x)
    \zeta_{\rho,2\rho}(x)
    \unk(x,t)
    \dgw(x)
    \\
    =
    -
    \int_{0}^{t}
    \sum_{x\in V}
    \abs{\dif_{y}\unk(x,\tau)}^{p-2}
    \dif_{y}\unk(x,\tau)
    \dif_{y}[\phi(x)\zeta_{\rho,2\rho}(x)\abs{x}^{\alpha}]
    \wgt(x,y)
    \di\tau
    \,.
  \end{multline*}
  In last integral, the term originating from
  $\dif_{y} \zeta_{\rho,2\rho}$ is seen to become vanishingly small as
  $\rho\to\infty$, since $\alpha<1$, similarly to what we did to bound
  $K_{2}$ in Section~\ref{s:l1}. Thus in the limit $\rho\to\infty$ we get
  \begin{equation*}
    \begin{split}
      &\sum_{x\not\in B_{2R}}
      \abs{x}^{\alpha}
      \unk(x,t)
      \dgw(x)
      \\
      &\quad\le
      -
      \int_{0}^{t}
      \sum_{x\in V}
      \abs{\dif_{y}\unk(x,\tau)}^{p-2}
      \dif_{y}\unk(x,\tau)
      \dif_{y}[\phi(x)\abs{x}^{\alpha}]
      \wgt(x,y)
      \di\tau
      \\
      &\quad=
      \int_{0}^{t}
      \sum_{x\in V}
      \abs{\dif_{y}\unk(x,\tau)}^{p-2}
      \dif_{y}\unk(x,\tau)
      \dif_{y}\zeta_{R,2R}(x)
      \abs{y}^{\alpha}
      \wgt(x,y)
      \di\tau
      \\
      &\qquad-
      \int_{0}^{t}
      \sum_{x\in V}
      \abs{\dif_{y}\unk(x,\tau)}^{p-2}
      \dif_{y}\unk(x,\tau)
      \dif_{y}\abs{x}^{\alpha}
      \phi(x)
      \wgt(x,y)
      \di\tau
      =:
      Q_{1}+Q_{2}
      \,.
    \end{split}
  \end{equation*}
  Since if $x\sim y$, $x\not\in B_{R}$,
  \begin{equation*}
    \abs{\dif_{y}\abs{x}^{\alpha}}
    \le
    \alpha
    \min(\abs{x},\abs{y})^{\alpha-1}
    \le
    \gamma
    R^{\alpha-1}
    \,,
  \end{equation*}
  we have
  \begin{equation*}
    \abs{Q_{1}}
    +
    \abs{Q_{2}}
    \le
    \gamma
    R^{\alpha-1}
    \int_{0}^{t}
    \sum_{x\in V}
    \abs{\dif_{y}\unk(x,\tau)}^{p-1}
    \wgt(x,y)
    \di\tau
    \,.
  \end{equation*}
  We bound the last integral by means of \eqref{eq:l1_nnn}, concluding as follows:
  \begin{multline}
    \label{eq:bbl_ik}
    \sum_{x\not\in B_{2R}}
    \unk(x,t)
    \dgw(x)
    \le
    R^{-\alpha}
    \sum_{x\not\in B_{2R}}
    \abs{x}^{\alpha}
    \unk(x,t)
    \dgw(x)
    \\
    \le
    \gamma
    R^{-1}
    t^{\frac{1}{p}}
    \norma{\unk_{0}}{\ell^{1}(V)}^{\frac{2(p-1)}{p}}
    \dcf_{1}^{(-1)}(t^{-1}\norma{\unk_{0}}{\ell^{1}(V)}^{-(p-2)})^{\frac{p-2}{p}}
    \le
    \gamma \varGamma^{-1}
    \norma{\unk_{0}}{\ell^{1}(V)}
    \,,
  \end{multline}
  where we have selected
  \begin{equation}
    \label{eq:bbl_jk}
    R
    \ge
    R_{p}(\unk_{0},t)
    :=
    \varGamma
    t^{\frac{1}{p}}
    \norma{\unk_{0}}{\ell^{1}(V)}^{\frac{p-2}{p}}
    \dcf_{1}^{(-1)}(t^{-1}\norma{\unk_{0}}{\ell^{1}(V)}^{-(p-2)})^{\frac{p-2}{p}}
    \,,
  \end{equation}
  for a $\varGamma>0$. This together with conservation of mass \eqref{eq:l1_n} proves \eqref{eq:bbl_n}, upon an 
  unessential redefinition of $R$.

  In order to prove \eqref{eq:bbl_p} we remark that from the argument above it follows that for $R$ as in \eqref{eq:bbl_jk},
  \begin{multline*}
    \sum_{x\in V}
    \abs{x}^{\alpha}
    \unk(x,t)
    \dgw(x)
    \le
    (2R)^{\alpha}
    \sum_{x\in B_{2R}}
    \unk(x,t)
    \dgw(x)
    \\
    +
    \sum_{x\not\in B_{2R}}
    \abs{x}^{\alpha}
    \unk(x,t)
    \dgw(x)
    \le
    \gamma
    R^{\alpha}
    \norma{\unk_{0}}{\ell^{1}(V)}
    \,,
  \end{multline*}
  where we have used conservation of mass again.
\end{proof}

\begin{proof}[Proof of Corollary~\ref{p:bbl2}]
  For a suitable choice of $\varGamma$, setting $R=2R_{p}$, $R_{p}$ as in \eqref{eq:bbl_jk},
  we have from \eqref{eq:bbl_ik}
  \begin{equation}
    \label{eq:bbl_ikk}
    \norma{\unk(t)}{\ell^{\infty}(V)}
    \msw(B_{R}(t))
    \ge
    \norma{\unk(t)}{\ell^{1}(B_{R}(t))}
    \ge
    \frac{1}{2}
    \norma{\unk_{0}}{\ell^{1}(V)}
    \,.
  \end{equation}
  The statement in \eqref{eq:bbl_nnn} then follows, if
  \eqref{eq:fkf_bbl} is assumed, on invoking Lemma~\ref{l:bbl}.
\end{proof}

\section{Proof of Theorem~\ref{t:decay}}
\label{s:decay}

We follow here ideas from \cite{Andreucci:1997},
\cite{Andreucci:Cirmi:Leonardi:Tedeev:2001}, \cite{Afanaseva:Tedeev:2004}.
Let $\unk_{R}$ be the solution to \eqref{eq:pde} with initial data
\begin{equation*}
  \unk_{R}(x,0)
  =
  \unk_{0}(x)
  \chi_{B_{R}(x_{0})}(x)
  \,,
  \qquad
  x\in V
  \,.
\end{equation*}
Then mass conservation and \eqref{eq:linf_int_m} with $r=1$ imply
\begin{equation}
  \label{eq:decay_i}
  \norma{\unk_{R}(t)}{\ell^{\infty}(V)}
  \le
  \gamma
  m_{R}
  \dcf_{1}^{(-1)}\big(t^{-1}m_{R}^{-(p-2)}\big)
  \,,
  \qquad
  t>0
  \,,
\end{equation}
where
\begin{equation*}
  m_{R}
  =
  \sum_{x\in B_{R}(x_{0})}
  \unk_{0}(x)
  \dgw(x)
  \,.
\end{equation*}
Let us also define $\unkiii_{R}=\unk-\unk_{R}$; note that
$\unkiii_{R}\ge 0$ by Proposition~\ref{p:compare}. In spite of the fact that
$\unkiii_{R}$ does not solve \eqref{eq:pde} we may still prove the
following inequality for $h\ge 0$, $t>\tau_{1}>\tau_{2}>0$, also by appealing to
Lemma~\ref{l:monot}:
\begin{multline}
  \label{eq:decay_ii}
  \sup_{\tau_{1}<\tau<t}
  \sum_{x\in V}
  \ppos{\unkiii_{R}(x,\tau)-h}^{q}
  \dgw(x)
  +
  \int_{\tau_{1}}^{t}
  \sum_{x,y\in V}
  \Abs{\dif_{y}\ppos{\unkiii_{R}(x,\tau)-h}^{\frac{p+q-2}{p}}}^{p}
  \wgt(x,y)
  \di \tau
  \\
  \le
  \frac{\gamma}{\tau_{1}-\tau_{2}}
  \int_{\tau_{2}}^{t}
  \sum_{x\in V}
  \ppos{\unkiii_{R}(x,\tau)-h}^{q}
  \dgw(x)
  \di\tau
  \,.
\end{multline}
As already observed in Remark~\ref{r:cacc} this is enough for us to
apply Proposition~\ref{p:linf_meas} and thus
Corollary~\ref{co:linf_int} to $\unkiii_{R}$, and get
\begin{equation}
  \label{eq:decay_iii}
  \norma{\unkiii_{R}}{\ell^{\infty}(V)}
  \le
  \gamma
  E_{q}^{\frac{1}{q}}
  \big[\dcf_{q}^{(-1)}
  \big(
  t^{-1}
  E_{q}^{-\frac{p-2}{q}}
  \big)
  \big]^{\frac{1}{q}}
  \le
  \gamma
  E_{q0}^{\frac{1}{q}}
  \big[\dcf_{q}^{(-1)}
  \big(
  t^{-1}
  E_{q0}^{-\frac{p-2}{q}}
  \big)
  \big]^{\frac{1}{q}}
  \,,
\end{equation}
where by invoking a simple variant of \eqref{eq:decay_ii} with $h=0$ we find
\begin{equation*}
  E_{q}
  :=
  \sup_{0<\tau<t}
  \sum_{x\in V}
  \unkiii_{R}(x,\tau)^{q}
  \le
  E_{q0}
  :=
  \sum_{x\not\in B_{R}(x_{0})}
  \unk_{0}(x)^{q}
  \dgw(x)
  \,.
\end{equation*}
We use here also Remark~\ref{r:dcf}.

Thus we have that for all $R>0$, since $\unk=\unk_{R}+\unkiii_{R}$,
\begin{equation}
  \label{eq:decay_j}
  \norma{\unk(t)}{\ell^{\infty}(V)}
  \le
  \gamma
  \Big\{
  m_{R}
  \dcf_{1}^{(-1)}\big(t^{-1}m_{R}^{-(p-2)}\big)
  +
  E_{q0}^{\frac{1}{q}}
  \big[\dcf_{q}^{(-1)}
  \big(
  t^{-1}
  E_{q0}^{-\frac{p-2}{q}}
  \big)
  \big]^{\frac{1}{q}}
  \Big\}
  \,.
\end{equation}
The first term on the right hand side of \eqref{eq:decay_j} is
increasing in $R$, while the second one is decreasing.  We aim at
making them equal, but this is in general impossible in the discrete
setting of graphs. We instead select $R$ as any number (optimally the
minimum one) such that
\begin{equation}
  \label{eq:decay_kk}
  m_{R}
  \dcf_{1}^{(-1)}\big(t^{-1}m_{R}^{-(p-2)}\big)
  \ge 
  E_{q0}^{\frac{1}{q}}
  \big[\dcf_{q}^{(-1)}
  \big(
  t^{-1}
  E_{q0}^{-\frac{p-2}{q}}
  \big)
  \big]^{\frac{1}{q}}
  \,.
\end{equation}
Then \eqref{eq:decay_n} is proved under assumption \eqref{eq:decay_kk}.

We need to make \eqref{eq:decay_kk} explicit. First, we define
\begin{equation*}
  X_{1}
  =
  \dcf_{1}^{(-1)}(t^{-1}m_{R}^{-(p-2)})
  \,,
  \qquad
  X_{q}
  =
  \dcf_{q}^{(-1)}(t^{-1}E_{q0}^{-\frac{p-2}{q}})
  \,,
\end{equation*}
so that from the definition of
$\dcf_{r}$, we get
\begin{equation}
  \label{eq:decay_k}
  X_{1}
  =
  t^{-\frac{1}{p-2}}
  m_{R}^{-1}
  \fkf(X_{1}^{-1})^{-\frac{1}{p-2}}
  \,,
  \qquad
  X_{q}^{\frac{1}{q}}
  =
  t^{-\frac{1}{p-2}}
  E_{q0}^{-\frac{1}{q}}
  \fkf(X_{q}^{-1})^{-\frac{1}{p-2}}
  \,.
\end{equation}
Therefore \eqref{eq:decay_kk} can be written as
\begin{equation}
  \label{eq:decay_kkk}
  \fkf(X_{1}^{-1})
  \le
  \fkf(X_{q}^{-1})
  \,,
  \quad
  \text{that is}
  \quad
  X_{1}
  \le
  X_{q}
  \,.
\end{equation}
We apply $\dcf_{1}$ and write the last inequality in the form
\begin{equation*}
  \begin{split}
    (tm_{R}^{p-2})^{-1}
    &\le
    \dcf_{1}(X_{q})
    =
    X_{q}^{p-2}
    \fkf(X_{q}^{-1})
    \\
    &=
    X_{q}^{\frac{(p-2)(q-1)}{q}}
    \dcf_{q}(X_{q})
    =
    X_{q}^{\frac{(p-2)(q-1)}{q}}
    (tE_{q0}^{\frac{p-2}{q}})^{-1}
    \,.
  \end{split}
\end{equation*}
From here we immediately get, on recalling the definition of $\dcf_{q}$,
\begin{equation*}
  \frac{1}{t}
  \ge
  \Big[
  \frac{E_{q0}}{m_{R}}
  \Big]^{\frac{p-2}{q-1}}
  \fkf\Big(
  {m_{R}^{\frac{q}{q-1}}}
  {E_{q0}^{-\frac{1}{q-1}}}
  \Big)
  \,.
\end{equation*}
This amounts to \eqref{eq:decay_nn} concluding the proof.

\bibliographystyle{abbrv}
\bibliography{paraboli,pubblicazioni_andreucci}
\end{document}